\documentclass[english,12pt,letterpaper]{article}
\usepackage[letterpaper,margin=1in]{geometry}\usepackage[utf8]{inputenc}
\usepackage[T1]{fontenc}
\usepackage{amsmath}
\usepackage{amssymb}
\usepackage{amsfonts}
\usepackage{amsthm}
\usepackage{algorithm}
\usepackage{float}

\usepackage{parskip}
\usepackage{graphicx}
  \usepackage{tikz}
\usetikzlibrary{decorations.pathreplacing,calc,positioning}
\def\DEFAULTFONT{}

\usepackage{base}
\usepackage{macros}
\usepackage{minimacros}
\DeclareMathOperator{\lcm}{lcm}
\newcommand{\Fp}{\mathbb{F}_p}
\makeatletter

\@for\@tempa:=lemma,corollary,claim,assumption,example,conjecture,question,fact,observation,open,definition,remark\do{%
  \expandafter\edef\csname theH\@tempa\endcsname{\noexpand\thesection.\noexpand\arabic{theorem}}}
\def\theHALG@line{\thealgorithm.\arabic{ALG@line}}
\makeatother
\newcommand{\tool}[1]{\cref{sec:known}\eqref{#1}}

\algtext*{EndIf}
\algtext*{EndFor}
\algtext*{EndProcedure}
\title{Faster Deterministic Integer Root Finding for Integer Polynomials\thanks{An extended abstract of this paper will appear in the Proceedings of the 2027 ACM-SIAM Symposium on Discrete Algorithms (SODA 2027).}}

\begin{document}
\author{Itamar Nir\thanks{This research was funded by the European Union (ERC, EACTP, 101142020). Views and opinions expressed are however those of the author(s) only and do not necessarily reflect those of the European Union or the European Research Council Executive Agency. Neither the European Union nor the granting authority can be held responsible for them.}}
\date{}
\maketitle
\begin{abstract}
We give a deterministic algorithm for finding all integer roots of a square-free polynomial $f\in\mathbb Z[x]$ of degree $n$ with $\lVert f\rVert_\infty<2^b$. The running time is
$$
\tilde{O}(n^{3/2}b),
$$
improving the $\tilde{O}(n^2b)$ bound of Harvey and Hittmeir (Research in Number Theory, 2022).

The algorithm follows the classical $p$-adic framework: find roots modulo a prime $p$, lift them modulo a high power of $p$, and verify the lifted candidates. The main new idea is to avoid searching for a prime for which $f\bmod p$ is square-free. Instead, we find a prime for which the total multiplicity of repeated roots modulo $p$ is small. This requires lifting repeated roots, which we handle using a weighted lifting tree. We also give a faster deterministic candidate-verification algorithm: given $n$ candidate integers smaller in absolute value than $2^b$, we decide which are roots of $f$ in
$$
\tilde{O}(nb+\min(n^2,nb^2))
$$
bit operations. Together, these ingredients give the first deterministic subquadratic-in-$n$ improvement for integer root finding in the square-free case.

\end{abstract}

\newpage
\tableofcontents
\newpage
 
\pagenumbering{arabic}
\setcounter{page}{1}

\section{Introduction}

Univariate polynomials over the integers are among the most basic objects in symbolic computation. We study the following natural problem. Given
\[
        f(x)=a_0+a_1x+\cdots+a_nx^n \in \mathbb Z[x],
\]
with coefficients of absolute value less than \(2^b\), we want to determine all
roots of \(f\) that lie in \(\mathbb Z\). In the dense representation the input
has \(\Theta(nb)\) bits, and the output contains at most \(n\) integers, each of absolute value less than $2^b$. Indeed, write $f=x^kg$ with $g(0)\neq 0$; every nonzero integer root of $f$ is a root of $g$, and hence divides the constant term $g(0)$, which is a nonzero coefficient of $f$. Thus
the natural bit-complexity benchmark is near-linear time in the input size, up
to polylogarithmic factors.

Integer root finding is a special case of polynomial factorization over \(\mathbb Q\): an integer root \(\alpha\) corresponds to a linear factor \(x-\alpha\). However, extracting only the linear factors is a much more
specific task than computing the complete factorization. Although polynomial factorization over \(\mathbb Q\) has been polynomial-time since the work of Lenstra, Lenstra, and Lov\'asz~\cite{LLL}, the fine-grained complexity of
extracting just the integer roots remains a separate and natural question.

Besides being a natural primitive, integer root finding is closely connected to deterministic algorithms for integer factorization and related divisor-finding problems. In Coppersmith-type methods, one constructs
auxiliary polynomials whose small integer roots encode the desired arithmetic information. Thus, after the lattice or modular step has produced such a polynomial, the remaining task is to recover its integer roots. See e.g.,~\cite{Coppersmith,BDHG,HarveyHittmeirRPower}.

Factorization and root finding algorithms often follow modular methods and Hensel lifting. In a typical modular approach
one reduces the polynomial modulo a prime \(p\), solves a problem over \(\mathbb F_p\), and then lifts the modular information to higher powers of \(p\). Such ideas underlie classical factorization algorithms, including the one by
Zassenhaus \cite{Zassenhaus}, and remain central in modern computer algebra, see~\cite{vzgg2013}. For the more special task of finding rational zeros of integral polynomials, Loos studied \(p\)-adic expansion methods and analyzed Newton--Hensel-type
algorithms~\cite{Loos}. His work already isolates rational root finding as a problem that can be attacked directly, without computing the complete
factorization.

More recently, Harvey and Hittmeir revisited this \(p\)-adic strategy in the
context of deterministic algorithms related to integer factorization, specifically to finding $r$-power divisors. As part
of their work, they gave a deterministic algorithm that finds all integer roots
of a square-free polynomial \(f\in\mathbb Z[x]\) in
\[
        \tilde{O}(n^2 b)
\]
bit operations, where \(n=\deg(f)\) and \(\|f\|_\infty<2^b\). Their algorithm follows the following three-step plan, which is also the basis of Loos's algorithm: find a prime \(p\) for which
$f_p := (f\bmod p)$ is square-free, find the roots of $f_p$, lift them by
Hensel lifting, and finally verify which lifted candidates are genuine integer
roots.

If randomization is allowed, this three-step plan gives a simple $\tilde{O}(nb)$ time algorithm. It is outlined in~\cite[Remark~A.8]{HarveyHittmeirRPower}. We sketch it here, and state it formally and prove it in Appendix~\ref{app:randomized}. Choose uniformly random prime numbers in the range $[100(nb+n\log n), 200(nb+n\log n)]$ until we find a prime $p$ for which $\gcd(f_p,f_p')=1$, i.e.\ $f_p$ is square-free. Then find the roots of $f_p$ by computing $f_p(0), f_p(1),\dots,f_p(p-1)$ via multipoint evaluation using \tool{comp:multipoint_evaluation}. Lift them to roots of $f \bmod p^b$ by Hensel lifting, i.e., $p$-adic Newton iteration~\cite[Section~9.4]{vzgg2013}. Discard any lifted root $x_i$ with $x_i \ge 2^b$. To verify which of the lifted roots are roots of $f$, choose a uniformly random prime number $100n^2b \le q \le 200n^2b$. Compute $f(x_i) \bmod q$ for each lifted root $x_i$ via multipoint evaluation of $f \bmod q$ using \tool{comp:multipoint_evaluation}, and output the $x_i$ with $f(x_i) \equiv 0 \pmod q$. The output always contains all non-negative integer roots of $f$, and with high probability it contains no other integers. Running the algorithm again on $f(-x)$ finds the negative integer roots of $f$.

Thus, with randomization, integer roots can be found in $\tilde{O}(nb)$ time. Deterministically, however, the best known algorithm, that of Harvey and Hittmeir, runs in time $\tilde{O}(n^2b)$.\footnote{This is stated in \cite[Proposition~A.6]{HarveyHittmeirRPower} for $b\ge n$. Checking each candidate $x_i$ by evaluating $f(x_i)$ instead of dividing $f$ by $x-x_i$ gives $\tilde{O}(n^2b)$ in general.} The main question addressed in this paper is whether a deterministic algorithm can move significantly closer to the input-size benchmark.

As explained in \cite[Remark A.8]{HarveyHittmeirRPower}
this standard \(p\)-adic strategy has two bottlenecks. The first is finding a small prime \(p\) for which the reduction \(f\bmod p\) is well behaved. The usual requirement is that \(f\bmod p\) be square-free, since then every root modulo \(p\) has a unique Hensel lift. Deterministically finding such a prime is already too expensive for our target running time. In particular, in their work Harvey and Hittmeir tested all primes up to $O(nb)$ \cite[Proposition A.5]{HarveyHittmeirRPower}. The second bottleneck is the final verification step. After lifting the roots
modulo a high power of $p$, we obtain a list of at most $n$ candidates and must
decide which of them are integer roots of $f$. Evaluating $f$ at a single
$b$-bit candidate costs $\tilde{O}(nb)$, so evaluating at all candidates costs
$\tilde{O}(n^2b)$. Fast multipoint evaluation does not help here: each value
$f(x_i)$ can have $\Theta(nb)$ bits, so just writing down all $n$ values
already takes $\Theta(n^2b)$ bits.

Our main result is that, under the same square-free input assumptions, we give an algorithm that finds all integer roots of $f$ in time $\tilde{O}(n^{3/2}b)$.
Similar to the preceding $p$-adic algorithms, our method is based on a modular strategy. However, we require the following new ideas in order to avoid the bottlenecks described above.

For the choice of the prime, we relax the condition that $f_p := f\bmod p$ must be square-free in $\mathbb{F}_p$.
Instead, we find a small prime
$p$ for which \emph{the sum of multiplicities of repeated roots of $f_p$ in $\mathbb{F}_p$ is small}: more
precisely, $\deg(\gcd(f_p,f'_p))$ is bounded by a parameter $t$. We make use of the properties of the resultant $\operatorname{res}(f,f')$ to show that a naive iteration over the prime numbers finds such $p$ efficiently. 
This weaker condition means that the usual simple-root form of Hensel lifting is no longer sufficient. We therefore develop a lifting procedure that handles repeated roots.

For the verification step, we give a deterministic multipoint method that,
given $n$ candidate integers $x_1,\dots,x_n$ with $|x_i| <2^b$, decides which of them are
roots of $f$ in
\[
\tilde{O}\bigl(nb+\min\{n^2,nb^2\}\bigr)
\]
bit operations. We give two algorithms: one running in $\tilde{O}(nb+n^2)$
time, and one running in $\tilde{O}(nb^2)$ time. Both algorithms use simple divisibility tests to reduce the candidate set to a size that can be checked
naively. The first repeatedly tests divisibility conditions of the form
$x_i-y\mid f(y)$ for small integers $y$. The second algorithm checks for each candidate $x_i$ whether $a_0 ^b \mid f(x_i)$ where $a_0$ is the constant term of $f$.

Combining these ingredients, and taking $t \simeq \sqrt{n}$, gives a
deterministic algorithm for finding all integer roots of $f$ in
$\tilde{O}(n^{3/2}b)$ time. We now state our results formally.

\subsection{Our results}

Our main result is an algorithm for finding the integer roots of a polynomial. Let $\|f\|_{\infty}$ denote the largest absolute value of a coefficient of $f$.

\begin{theorem}[Finding roots]\label{thm:finding_integer_roots}
    Let $f \in \mathbb{Z}[x]$ be a square-free polynomial of degree $n$ satisfying $\|f\|_\infty <2^b$ for some $b \in \mathbb{N}$. Then there is a deterministic algorithm that finds all integer roots of $f$ in time $\tilde{O}(n^{3/2}b)$.
\end{theorem}

We focus on integer roots in the main body of the paper. The same $p$-adic
lifting approach can also be adapted to rational roots by standard rational reconstruction; see
\cref{sec:future_work}.

The proof of this theorem relies on the following two results that are interesting on their own. The first gives an efficient way to compute sufficiently many $p$-adic digits of all $p$-adic roots (and integer roots in particular). A root of a polynomial is called a simple root if its multiplicity is equal to $1$, and otherwise it is called a repeated root. 

\begin{theorem}[Approximating $p$-adic roots]\label{thm:lifting_roots}
Let $n,b$ be positive integers and let $f \in \Z[x]$ be a polynomial of degree $n$. Let $p$ be a prime number that does not divide the leading coefficient of $f$. Let $t \ge 1$ be a given upper bound on the sum of multiplicities of the repeated roots of $f_p$ in $\mathbb{F}_p$, i.e., roots of multiplicity  greater than one. Assume that $\|f\|_\infty \le p^{bt}$. Then we may find a set of at most $n$ residue classes modulo $p^b$,
represented by non-negative integers \(<p^b\), that contains the reduction
modulo $p^b$ of every root of $f$ in $\mathbb{Z}_p$, in time $\tilde{O}(nbt + (t+\sqrt{n})\sqrt{p})$.
\end{theorem}

 In our proof of \Cref{thm:finding_integer_roots}, the parameter $t$ will eventually be set to $\lfloor\sqrt{n}\rfloor$.

Another contribution that is interesting on its own is the following fast deterministic verification procedure.

\begin{theorem}[Verifying roots]\label{thm:verifying_roots}
Let $n,b$ be positive integers and let $f \in \Z[x]$ be a polynomial of degree $n$, satisfying $\|f\|_{\infty} < 2^b$. In addition, let $x_1,\ldots,x_m$ be distinct integers with $|x_i| < 2^b$ and $m \le n$. Then we may determine for each $1\le i \le m$ whether $f(x_i)=0$ in $\tilde{O}(nb+\min(n^2,nb^2))$ time.
\end{theorem}

To apply these theorems we shall need an appropriate prime number $p$ such that no root has too large a multiplicity modulo $p$, along with some other technical properties. This is given in the following theorem.

\begin{theorem} \label{thm:finding_p}
Let $n,b$ be positive integers and let $f \in \Z[x]$ be a square-free polynomial of degree $n$, satisfying $\|f\|_{\infty} < 2^b$.  Let $1\le t \le n$ be an additional integer parameter. Then we may find a prime number $p$ such that the following three requirements are satisfied:
\begin{enumerate}
    \item $p \le \tilde{O}(\frac{nb}{t})$ \label{req:p_is_small}
    \item $\deg(\gcd(f_p,f_p')) < t$ \label{req:multiplicities_are_small}
    \item $p$ does not divide the leading coefficient of $f$\label{req:pndiv_a_n}
\end{enumerate}
in $\tilde{O}(\frac{n^2b}{t})$ time.

\end{theorem}

\subsection{Proof overview}

We now give an overview of the ideas that go into the proofs of our results. Schematically, the structure of the proof is:

\begin{center}
\begin{tikzpicture}[
    node distance=1.7cm,
    box/.style={
        draw,
        rounded corners,
        align=center,
        minimum width=3.2cm,
        minimum height=0.9cm,
        font=\small
    },
    arrow/.style={->, thick}
]
\node[box] (p) {Find a good prime $p$\\
$\deg\gcd(f_p,f_p')< t$};
\node[box, right=of p] (lift) {Lift roots\\
to precision $p^b$};
\node[box, right=of lift] (verify) {Verify candidates\\
over $\mathbb Z$};

\draw[arrow] (p) -- (lift);
\draw[arrow] (lift) -- (verify);
\end{tikzpicture}
\end{center}

\paragraph{Finding $p$ (\cref{thm:finding_p}).} Given a parameter $1\le t\le n$, this theorem lets us find a prime
$
    p \le \widetilde O\!\left(\frac{nb}{t}\right)
$
that does not divide the leading coefficient of $f$ and satisfies
\[
    \deg\bigl(\gcd(f_p,f_p')\bigr)<t.
\]

The main point is to guarantee the last inequality. For this purpose, we use the resultant. We prove in \cref{lem:res_modular_bound} that, for every prime $p$,
\[
    \deg\bigl(\gcd(f_p,f_p')\bigr)
    \le v_p\bigl(\operatorname{res}(f,f')\bigr).
\]
Since $f$ is square-free, the resultant is nonzero, and a standard bound, \cref{lem:res_size_bound}, gives $
    |\operatorname{res}(f,f')|\le 2^{\widetilde O(nb)}.
$
Therefore, by \cref{lem:not_lots_of_bad_primes}, the inequality $\deg(\gcd(f_p,f_p'))\ge t$
holds for at most $\widetilde O(nb/t)$ primes. The leading coefficient has at most $b$ prime divisors, which is also $\widetilde O(nb/t)$ since $t\le n$. Hence, by trying all primes up to a sufficiently large bound of size $\widetilde O(nb/t)$, we are guaranteed to find a prime satisfying all the requirements of \cref{thm:finding_p}. This is \cref{alg:algorithm_find_p}.

\paragraph{Lifting roots (\cref{thm:lifting_roots}).}
Given the prime $p$, we want to approximate the roots of $f$ in $\mathbb Z_p$, where $\bZ_p$ denotes the ring of $p$-adic integers. More precisely, we wish to compute a list of at most $n$ residue classes modulo $p^b$, represented by non-negative integers smaller than $p^b$, such that every root of $f$ in $\mathbb Z_p$ reduces to one of these classes. For non-negative integer roots this gives actual candidates: if $x$ is a non-negative integer root, then $x<2^b\le p^b$, so the residue representative modulo $p^b$ is just $x$. Negative integer roots are handled by applying the same procedure to $f(-x)$. The resulting candidates are then passed to \cref{thm:verifying_roots}.

To prove \cref{thm:lifting_roots}, we construct a lifting tree. It is best to think of its nodes as base-$p$ digit strings. A node in layer $m$ represents a prefix of length $m$, with value
\[
    x_0 = \delta_0+\delta_1p+\cdots+\delta_{m-1}p^{m-1},
\]
where the $\delta_i$ are $p$-adic digits, i.e., integers in $\{0,\ldots,p-1\}$.
Its children have the same first $m$ digits and one additional digit in position $m$. Descendants of a node represent possible $p$-adic lifts of that prefix.

The root is represented by the empty string $\epsilon$. Its value is $0$, but we distinguish  the empty string $\epsilon$ from the one-digit string $0$. For a node $x_0$ in layer $m$, define
\[
    f_{x_0}(a):=f(x_0+p^m a),
\]
where $a$ is a new variable. If a $p$-adic root has prefix $x_0$, then it has the form
\[
    x=x_0+p^m \beta
\]
for some $\beta\in \mathbb Z_p$, and therefore
\[
    f_{x_0}(\beta)=f(x_0+p^m \beta)=f(x)=0.
\]
Let $p^{v(f_{x_0})}$ denote the largest power of $p$ that divides all coefficients of $f_{x_0}$.
Dividing $f_{x_0}$ by  $p^{v(f_{x_0})}$  we obtain the polynomials
\begin{align*}   
(f_{x_0})^* &:= \frac{f_{x_0}}{p^{v(f_{x_0})}} ,\\
(f_{x_0})_p^* &:= \frac{f_{x_0}}{p^{v(f_{x_0})}} \bmod p.
\end{align*}
Note that the next digit
\[
    \delta = \beta\bmod p
\]
satisfies
\[
    (f_{x_0})_p^*(\delta)=0
\]
over $\mathbb F_p$. For every such digit $\delta$, we add the child represented by the digit string with value $x_0+p^m\delta$, and label the edge by $\delta$.

We next bound the size of the lifting tree. For this purpose, we assign weights to nodes. Let $x\ne \epsilon$ be a node with parent $y$, reached from $y$ by the digit $\delta$. The weight $w(x)$ is the multiplicity of $\delta$ as a root of $(f_y)_p^*$. For the root we set
\[
    w(\epsilon):=\deg (f_\epsilon)_p^*=\deg f_p^*.
\]

An important property that we prove in \cref{lem:deg_le_weight}, which is crucial for controlling the size of the tree, is
\[
    \deg (f_x)_p^* \le w(x).
\]
It follows that the sum of the weights of the children of a node $x$ is at most $w(x)$, as shown in \cref{lem:weights_of_children_sum_bound}: the children correspond to distinct roots of $(f_x)_p^*$, and the sum of their multiplicities is at most the degree of $(f_x)_p^*$.

By induction on the layers, the total weight of every layer is at most
\[
    w(\epsilon)=\deg f_p^*\le n.
\]
Since every node has positive weight, every layer contains at most $n$ nodes, see \cref{lem:tree_layers_have_atmost_n_nodes}.

This size bound alone is not enough for an efficient algorithm. The difficulty is that, for each node, we have to compute its branching polynomial $f_{x_0}^*$, and the coefficients of these polynomials can be very large. We therefore prove that only bounded $p$-adic precision is needed. We say that a polynomial is known to precision $k$ if it is known modulo $p^k$.

The precision statement we use is the following, which follows from \cref{lem:tree_equations_dont_have_large_vp,lem:coefficients_matter_only_mod_p^sb}: to determine the next $k$ layers in the subtree below a node $x$, it is enough to know the normalized branching polynomial $f_x^*$ to precision $w(x)\cdot k$.

We then give an efficient algorithm for constructing layers $0$ through $b$ of the lifting tree, whose running time is given in \cref{lem:build_tree_cost}. The basic operation is this: suppose that we know the normalized branching polynomial $f_{x_1}^*$ at a node $x_1$, and we know the digits along the path from $x_1$ to a descendant $x_2$. Then, by \cref{lem:computing_branching_polys}, we can compute the normalized branching polynomial $f_{x_2}^*$ losing only a controlled amount of precision.

Figures~\ref{fig:spine} (page~\pageref{fig:spine}), \ref{fig:tree-recursion} (page~\pageref{fig:tree-recursion}) and \ref{fig:finishing} (page~\pageref{fig:finishing}) illustrate the flow of the tree construction.

The construction, \cref{alg:Tree building algorithm}, is by divide and conquer on the number of layers. First we construct the first $b/2$ layers of the lifting tree of $f$. For each node $x$ in the middle layer, the descendants of $x$ are determined by $f_x^*$. Thus, once $f_x^*$ is known to sufficient precision, we may recursively construct the remaining layers below $x$. This gives the first $b$ layers of the lifting tree of $f$.

It remains to solve the following problem efficiently: given the first $k$ layers of the lifting tree, compute the branching polynomials of all nodes in layer $k$ to the required precision. \Cref{alg:computing_branching_polys} solves this problem.

Computing all layer-$k$ branching polynomials independently would be too expensive. Instead, we decompose the known tree into a heavy path, which we call the \emph{spine}, and a collection of light subtrees branching off it. We compute the branching polynomials along the spine using a divide-and-conquer procedure, \cref{alg:computing_spine_polys}, and then recurse on the light subtrees. The weight inequalities above, together with \cref{lem:A_weight}, ensure that the total work in the light subtrees can be charged to their weights, giving the required running time.

Finally, \cref{alg:finish_lifting_roots} proves \cref{thm:lifting_roots}. It finds the roots of $f_p$ and their multiplicities using \cref{lem:find_roots_and_multiplicities}, lifts the simple roots by Hensel lifting, and runs \cref{alg:Tree building algorithm} below each repeated root.

\paragraph{Verifying roots (\cref{thm:verifying_roots}).}
We are given $m$ candidates $x_1,\dots,x_m$ and need to decide which of them are roots of $f$. We give two verification algorithms. The first runs in $\widetilde O(n^2+nb)$ time and is best when $b$ is large. The second runs in $\widetilde O(nb^2)$ time and is best when $b$ is small.

The algorithms assume, after the standard reductions in \cref{sec:verifying}, that the candidates are positive and that the constant coefficient $a_0$ is nonzero.

The first algorithm, \cref{alg:algorithm_verify_large_b}, works as follows. We first check all small candidates, namely candidates with
\[
    x_i \le 32n^2 2^{b/n},
\]
using multipoint evaluation (see \tool{comp:multipoint_evaluation}), and then remove them from the candidate list. Next, using \cref{lem:find_segment_f_doesn't_vanish}, we find an interval $[a,a+2n-1]$, with $a\le 2n^2$, on which $f$ does not vanish. We then iterate over $y=a,\dots,a+2n-1$ and, for all remaining candidates $x_i$ simultaneously, test whether
\[
    x_i-y \mid f(y)
\]
using multipoint-modulo (see \tool{comp:multipoint_modulo}). Candidates that fail one of these tests are eliminated. Each value of $y$ gives one elimination round.

The correctness argument is based on the congruence
\[
    f(x_i)\equiv f(y) \pmod {x_i-y}.
\]
Thus, if $x_i$ survives the test for a given $y$, then $x_i-y$ divides $f(x_i)$. If $x_i$ survives all tests, then
\[
    \operatorname{lcm}\bigl(x_i-a, x_i-(a+1),\ldots,x_i-(a+2n-1)\bigr)
    \mid f(x_i).
\]
We prove in \cref{sec:verify_large_b_correctness} that, unless $x_i$ is one of the small candidates already removed, this least common multiple is larger than $|f(x_i)|$. Hence any remaining candidate must satisfy $f(x_i)=0$.

For the running time, we prove that the total bit size of the surviving candidates decreases quickly during the elimination rounds. More precisely, by \cref{lem:round_product}, if $S_r$ is the set of candidates entering round $r$, then
\[
    \sum_{z\in S_r}\log z = \widetilde O\!\left(\frac{nb}{r}\right).
\]
Therefore (using multipoint modulo \tool{comp:multipoint_modulo}) round $r$ costs
\[
    \widetilde O\!\left(n+b+\frac{nb}{r}\right),
\]
and summing over $r=1,\dots,2n$ gives
\[
    \sum_{r=1}^{2n} \widetilde O\!\left(n+b+\frac{nb}{r}\right)
    = \widetilde O(n^2+nb),
\]
using the harmonic series.

The algorithm for small $b$, \cref{alg:algorithm_verify_roots_small_b}, is simpler. It first tests whether
\[
    x_i \mid a_0
\]
and then tests whether
\[
    |a_0|^b \mid f(x_i).
\]
Only the candidates surviving both tests are evaluated individually using \tool{comp:substitute_in_poly}. Every root of $f$ passes both tests, so the algorithm is correct; see \cref{sec:verify_small_b_correctness}.

To prove the efficiency of this algorithm, we show in \cref{not_many_candidates_remain} that at most $b+1$ candidates survive the two tests. The proof uses lifting-tree bounds similar to those discussed earlier. The idea is that, if a survivor $x_i$ is divisible by a prime $p\mid a_0$, then the relevant prefixes of $x_i/p$ appear in the lifting tree of $f_0(y)=f(py)$. This implies, by \cref{x_appears_on_lifting_tree}, that for each prime $p\mid a_0$ the number of surviving candidates divisible by $p$ is at most $v_p(a_0)$. Since
\[
    \sum_{p\mid a_0} v_p(a_0) \le \log |a_0| \le b,
\]
there are at most $b$ such survivors, plus the possible candidate $x_i=1$. Thus only $b+1$ candidates are checked individually, giving the $\widetilde O(nb^2)$ bound.  

\subsection{Related work}

We discuss previous work on the two algorithmic ingredients of our approach:
finding roots over $\mathbb F_p$, and lifting them to $p$-adic precision.

\paragraph{Root finding over $\mathbb F_p$}
Deterministic root finding over finite fields is a classical problem; see, for
example, the survey in \cite[Section~1.3]{GrenetHoevenLecerf}. Shoup~\cite{Shoup1990}
gave a deterministic algorithm that factors a polynomial of degree $n$ over $\mathbb F_p$
in time $O(p^{1/2}(\log p)^2 n^{2+\varepsilon})$. Grenet, van der Hoeven and Lecerf give
deterministic algorithms based on Graeffe transforms~\cite{GrenetHoevenLecerf}, which
are particularly efficient when $p-1$ is smooth. For a general prime $p$, their
simultaneous baby-step giant-step algorithm \cite[Proposition~8]{GrenetHoevenLecerf}
finds the roots in $\mathbb F_p$ of a polynomial of degree $n$ in time
$\tilde{O}(\sqrt{np}+n)$, given a primitive root of $\mathbb F_p^*$. This bound is all we
need. For completeness, we include in \cref{sec:roots-Fp} a simpler self-contained
proof of it, also based on a baby-step giant-step argument, which does not require a
primitive root (\cref{comp:roots_mod_p}).

\paragraph{Lifting repeated roots}
When $f\bmod p$ is square-free, Newton iteration lifts all of its roots to
precision $p^b$ in quasi-linear time. The difficulty lies in repeated roots, which may
have many lifts and to which Newton iteration does not apply. The standard approach,
which we also use, is to describe the roots by a tree of prefixes of their expansions.
This idea underlies the root-finding steps of the Reed--Solomon list-decoding algorithms
of Roth and Ruckenstein~\cite[Section~V]{RothRuckenstein} and of
Alekhnovich~\cite[Appendix]{Alekhnovich}, which find the power series roots of a
polynomial to a given precision. Berthomieu, Lecerf and
Quintin~\cite[Algorithms~1 and~5]{BerthomieuLecerfQuintin} use the same approach over
a large class of complete local rings, including $\mathbb Z_p$. Over $\mathbb Z_p$,
their fastest algorithm finds the roots to precision $N$ using an expected number of
$\tilde{O}((N+\log p)Nn\log p)$ bit operations, with randomized root finding over
$\mathbb F_p$ \cite[Corollary~9]{BerthomieuLecerfQuintin}.
To recover the first $b$ digits of the roots when the repeated roots of $f \bmod p$
have total multiplicity at most $t$, one has to take $N$ of order $bt$ in the worst case,
which gives $\tilde{O}(nb^2t^2)$ apart from root finding over $\mathbb F_p$, compared with
$\tilde{O}(nbt)$ for our algorithm of \cref{thm:lifting_roots}. Over $K[[x]]$, Neiger, Rosenkilde and
Schost~\cite{NeigerRosenkildeSchost} refined Alekhnovich's divide-and-conquer approach
into an algorithm that finds all roots to precision $d$ using $\tilde{O}(dn)$
operations in $K$, apart from univariate root finding over $K$. They keep the
recursion small by replacing each intermediate polynomial with a polynomial of lower
degree that has the same roots. Our lifting tree is closely related to their
construction, but our setting differs in two ways. First, we work over $\mathbb Z_p$
rather than $K[[x]]$. Second, \cite{NeigerRosenkildeSchost} count operations in $K$ and
treat root finding over the residue field as a black box, whereas we count bit operations.
Instead of
lowering the degrees of the intermediate polynomials, we rely on the bound $t$ on the
multiplicities of the repeated roots, which our choice of $p$ keeps small.

\paragraph{Factorization over $\mathbb Z_p$}
A different approach is to factor $f$ over $\mathbb Z_p$ and read the roots off its
linear factors. The OM-based algorithm of Poteaux and Weimann~\cite{PoteauxWeimannMontes}
computes a full factorization efficiently, but separating all roots requires a precision
of order $v_p(\operatorname{disc}(f))$. This may be of order $nb/\log p$ even when
$f\bmod p$ has few repeated roots, which leads to a running time of order $n^2b$ in our
setting. Poteaux, Vaccon and Weimann~\cite[Theorem~5.9]{PoteauxVacconWeimann2026} compute an
irreducible factorization of $f+O(p^\sigma)$ at any precision $\sigma$, using
$\tilde{O}(n^2\sigma+n\sigma^2)$ operations in $\mathbb F_p$ plus some univariate
factorizations; with $\sigma\ge b$, this is again of order $n^2b$. They also use
Brink's continuity of roots theorem~\cite[Theorem~2]{Brink2006}, which bounds how far the roots of a
polynomial can move when its coefficients are perturbed, to locate the roots of all
polynomials of the form $f+O(p^\sigma)$; our
\cref{lem:tree_equations_dont_have_large_vp,lem:coefficients_matter_only_mod_p^sb}
play a similar role for the lifting tree. Over $K[[x]]$, van der Hoeven and
Lecerf~\cite[Theorem~1]{HoevenLecerfContact} compute irreducible factorizations at any
precision $\sigma$ using $\tilde{O}(n^{1+\varepsilon}\sigma)$ operations in $K$, for any
fixed $\varepsilon>0$, plus some univariate factorizations, provided the characteristic of
$K$ is $0$ or larger than $n$.

Under the assumption that the sum of multiplicities of the repeated roots of $f \bmod p$
is at most $t$, we construct the first $b$ layers of the lifting tree, and hence the
first $b$ digits of the roots of $f$ in $\mathbb Z_p$, in time
$$
\tilde{O}\bigl(nbt+(t+\sqrt n)\sqrt p\bigr).
$$
Thus, the novelty is not the tree viewpoint itself, but the efficient construction
of the relevant $p$-adic lifting tree.

\subsection{Organization}
The paper is organized as follows. In \Cref{sec:prelim} we introduce notation
and recall several standard algorithmic tools. In \Cref{sec:roots-Fp} we discuss algorithms for finding roots of polynomials over finite
fields and for computing their multiplicities. In \cref{sec:finding-p} we prove \cref{thm:finding_p}, which shows that
an appropriate prime $p$ can be found efficiently. Next, in \Cref{sec:lifting} we develop the lifting procedure for non-simple
roots modulo powers of $p$. The fast verification procedure is given in \Cref{sec:verifying}. We combine everything to prove \Cref{thm:finding_integer_roots} in \Cref{sec:final-proof}.

\section{Preliminaries}\label{sec:prelim}

\subsection{Notations}

Throughout this paper, $p$ will always denote some prime number. All logarithms are in base $2$, and $\ln$ denotes the natural logarithm.
For a polynomial $f \in \mathbb{Z}[x]$ and a prime number $p$, let $f_p \in \mathbb{F}_p[x]$ denote the reduction of $f$ modulo $p$. For a non-zero integer $x$, $v_p(x)$ is defined as the largest integer $k$ such that $p^k \mid x$. For a polynomial $f = a_0 + a_1x + \dots + a_nx^n \in \mathbb{Z}[x]$ and a prime $p$, we define $v_p(f) = \min_i(v_p(a_i))$. Let $\pi(n)$ denote the number of primes not larger than $n$. We use throughout the fact that $\pi(n) = \Theta(\frac{n}{\log n})$, which follows from the prime number theorem.

\subsection{Algorithmic tools}\label{sec:known}

Our algorithms will rely on several known tools from the literature, such as integer multiplication,
polynomial multiplication, division with remainder, gcds, product trees,
and multipoint evaluation. We list them below. Most of them can be found in the textbook
of von zur Gathen and Gerhard~\cite{vzgg2013}, and we give precise references below. Throughout, $\tilde{O}$ suppresses
polylogarithmic factors.

\begin{enumerate}
    \item\label{comp:polynomial_arithmetic} Two polynomials of degree $<n$ over a (commutative,
    unital) ring can be multiplied in $\tilde{O}(n)$ ring
    operations. They can also be divided with remainder, in $\tilde{O}(n)$ ring operations, given that the divisor has an invertible leading coefficient
    \cite{cantorkaltofen1991}, \cite[Theorems~8.23 and~9.6]{vzgg2013}.

    \item\label{comp:product_tree} Let $R$ be a commutative ring and $f_1,\dots,f_n \in R[x]$. Their product $f_1 f_2 \cdots f_n$ can be computed in $\tilde{O}(\sum \deg(f_i))$
    ring operations by the product tree: multiply the $f_i$ in pairs and recurse \cite[Lemma~10.4]{vzgg2013}.

    \item\label{comp:multipoint_modulo} Given an integer $x$ and moduli $m_1,\dots,m_k$, all
    residues $x \bmod m_i$ can be computed in $\tilde{O}(\log x + \sum_i \log m_i)$ time. Similarly, given a polynomial $f(x) \in R[x]$ and monic moduli $m_1(x),\ldots,m_k(x) \in R[x]$, the polynomials $f(x) \bmod m_i(x)$ can be computed in $\tilde{O}(\deg f + \sum_i \deg m_i)$ operations over $R$. The idea is to build a product tree of the moduli, and then traverse it starting with $f$
    at the root, taking modulo the relevant subproduct at each step \cite[Theorem~10.24 and Corollary~10.17]{vzgg2013}.

    \item\label{comp:multipoint_evaluation} A polynomial $f$ of degree $<n$ over a ring can be
    evaluated at $m$ points $a_1,\ldots,a_m$ in $\tilde{O}(n+m)$ ring operations. This is essentially \tool{comp:multipoint_modulo} on $f$ modulo $x-a_i$ \cite[Corollary~10.8]{vzgg2013}.

    \item\label{comp:gcd} The $\gcd$ of two polynomials of degree $\le n$ over $\Fp$ can be
    computed in $\tilde{O}(n)$ field operations, via the fast Euclidean algorithm
    \cite[Corollary~11.9]{vzgg2013}.

    \item\label{comp:substitute_in_poly} Let $f \in \mathbb{Z}[x]$ have degree $n$ with $b$-bit
    coefficients, and let $x \in \mathbb{Z}$ have $b$ bits. Then $f(x)$ can be computed in
    $\tilde{O}(nb)$ bit operations: writing $f = f_0 + x^{\lceil n/2 \rceil} f_1$
    and recursing yields $T(n) = 2T(n/2) + \tilde{O}(nb) = \tilde{O}(nb)$.

    \item\label{comp:taylor_shift} Let $R$ be a commutative ring, $f \in R[x]$ of degree $\le n$,
    and $a,b \in R$. Then $f(ax+b)$ can be computed in $\tilde{O}(n)$ ring operations: the polynomial
    $g(x)=f(x+b)$ can be computed in $\tilde{O}(n)$ ring operations \cite[Corollary~9.16]{vzgg2013},
    and $f(ax+b)=g(ax)$ is obtained from $g$ by multiplying its $k$th coefficient by $a^k$.

    \item\label{comp:sieve} The first $B$ primes can be found in $\tilde{O}(B)$ time using the
    sieve of Eratosthenes.
\end{enumerate}

\paragraph{Complexity model.} Unless stated otherwise, ``in time $T$'' means using $T$ bit operations. Some of the tools above are stated in terms of the number of arithmetic operations in a ring $R$; we use them for $R=\mathbb Z$, $R=\mathbb F_p$, and $R=\mathbb Z/p^k\mathbb Z$. Multiplication and division with remainder of two $N$-bit integers cost $\tilde{O}(N)$ time \cite[Section~2.3.3 and Theorem~1.4]{brent2010modern}. Hence, one arithmetic operation in $\mathbb F_p$, including inversion of a nonzero element via the fast extended Euclidean algorithm \cite[Corollary~11.13]{vzgg2013}, costs $\tilde{O}(\log p)$ bit operations, and one ring operation in $\mathbb Z/p^k\mathbb Z$ costs $\tilde{O}(k\log p)$ bit operations. Throughout, we refer to the items of the list above as \cref{sec:known}(1)--(8).

\section{Roots of polynomials over $\F_p$}\label{sec:roots-Fp}

In this section we explain how to find the roots of a polynomial in $\mathbb F_p[x]$, together with their multiplicities. We use this mainly in \cref{sec:lifting}.

\begin{lemma}
\label{comp:roots_mod_p}
Let $p$ be a prime number, and let $f \in \mathbb{F}_p[x]$ be a nonzero
polynomial of degree $n$. Then all roots of $f$ in $\mathbb{F}_p$ can be
found in time
\[
    \tilde{O}(\sqrt{pn}+n).
\]
\end{lemma}

\begin{proof}
Throughout this proof, the logarithmic factors coming from arithmetic in
$\mathbb F_p$ are absorbed into the $\tilde{O}$ notation.

If $n \geq p$, then we simply evaluate $f$ at all
points of $\mathbb{F}_p$ by \tool{comp:multipoint_evaluation}. This takes
$\tilde{O}(n+p)=\tilde{O}(n)$ time. Hence assume from now on that $n<p$.

Set
$
    k=\left\lceil \sqrt{\frac{p}{n}} \right\rceil.
$
For $0 \leq j < k$, compute the shifted polynomials $f(x+j)$
using \tool{comp:taylor_shift}. This takes total time $\tilde{O}(nk)$. Now form
\[
    g(x) := \prod_{j=0}^{k-1} f(x+j)
\]
by a product tree. Since $\deg g \leq nk$, this also takes
$\tilde{O}(nk)$ time.

Next, evaluate $g$ at the arithmetic progression $0, k, 2k, \ldots, \left\lfloor \frac{p-1}{k} \right\rfloor k$
using \tool{comp:multipoint_evaluation}. The number of evaluation points is
$O(p/k)$, so this step costs
\[
    \tilde{O}\left(nk+\frac{p}{k}\right).
\]
The idea is that each root of $g$ in $0,k,2k,\ldots$ tells us about a segment that contains a root of $f$.

Specifically, for each integer $a$ with $0 \leq a \leq \lfloor (p-1)/k \rfloor$, observe
that
\[
    g(ak)=\prod_{j=0}^{k-1} f(ak+j).
\]
Thus $g(ak)=0$ if and only if the interval $[ak,ak+k-1]$ contains a root of $f$. In the last interval we interpret the points modulo $p$. Mark each interval that contains a root of $f$. Since $f$ has at most $n$ roots in $\mathbb{F}_p$, and all intervals other than the last one are disjoint,
at most $n+1$ of the intervals can be marked in this way. Therefore the
union of all marked intervals contains at most $(n+1)k$ candidate points.

Finally, evaluate $f$ on all candidate points using \tool{comp:multipoint_evaluation}, and output those candidates at which the value is
zero. This last step takes
\[
    \tilde{O}(nk).
\]

Combining the costs, the total running time is
$
    \tilde{O}\left(nk+\frac{p}{k}\right).
$
By the choice of $k=\lceil \sqrt{p/n}\rceil$, the total running time is
\[
    \tilde{O}(\sqrt{pn}+n),
\]
as claimed.
\end{proof}

Now we describe how to find the multiplicities of the roots of $f$.

\begin{lemma}\label{lem:find_multiplicities}
    Let $p$ be a prime number, and let $f \in \mathbb{F}_p[x]$ be a nonzero polynomial of degree $n$. Suppose that distinct roots $x_1,\ldots,x_k$ of $f$ are given. Then we may find their multiplicities $m_1,\dots,m_k$ in time
    \[
        \tilde{O}(n\log(p)).
    \]
\end{lemma}

\begin{proof}
As the $x_j$ are distinct roots,
\begin{equation}\label{eq:mult-sum}
\sum_{j=1}^{k} m_j \le \deg f = n,
\end{equation}
which we use repeatedly. For a root $x_j$ and integer $d\ge 1$,

\begin{equation}\label{eq:root_multiplicity}
    (x-x_j)^d \mid f \iff m_j\ge d. 
\end{equation}

We use the following ability repeatedly. Given some integers $q_1,\dots,q_k$, we may check for each $j$ whether $m_j \ge q_j$ by seeing which of the values $f \bmod{(x-x_j)^{q_j}}$ are $0$. These values can be computed in $\tilde{O}(n+q_1+\cdots+q_k)$ field operations using \tool{comp:multipoint_modulo}. This allows us to perform a parallel binary search to find the $m_j$. We must however obtain approximations to the values $m_j$ before performing the parallel binary search, since we want to make sure that $q_1 + \dots +q_k = O(n)$ in each query. 

\medskip
\noindent\textbf{Step 1 (bracketing).} We compute, for each $j$, a power of two $l_j$ with $l_j \le m_j < 2 l_j$.
Initialize $S_0 = \{1,\dots,k\}$, maintaining the invariant
\[
S_i = \{\, j : m_j \ge 2^{i} \,\}.
\]
Given $S_i$, set $q_j = 2^{i+1}$ for $j \in S_i$ and $q_j = 0$ otherwise. Then we may compute the set $S_{i+1} = \{\, j \in S_i : m_j \ge 2^{i+1} \,\}$. We stop at the first $i$ with $S_{i+1}=\emptyset$,
which happens once $2^{i+1} > n$, i.e.\ after $O(\log n)$ steps. For each $j$, let $t_j$ be the last level
containing it and set $l_j := 2^{t_j}$, $r_j := 2 l_j$; then $l_j \le m_j < r_j$.

\emph{Cost.} Every $j \in S_i$ has $m_j\ge 2^i$, so by \eqref{eq:mult-sum}, $|S_i|\cdot 2^i \le n$ and hence
$q_1+\cdots+q_k= |S_i|\cdot 2^{i+1} \le 2n$. Each of the $O(\log n)$ iterations thus costs $\tilde{O}(n)$
field operations, resulting in $\tilde{O}(n\log p)$ time.

\medskip
\noindent\textbf{Step 2 (parallel binary search).} Starting from the brackets of Step 1, we maintain
$l_j \le m_j \le r_j$. Note $r_j \le 2 m_j$ initially, and since $r_j$ never increases this persists.
While some bracket has $l_j < r_j$, set $q_j = \lceil (l_j + r_j)/2 \rceil$ and update
\[
(l_j, r_j) \gets
\begin{cases}
(q_j,\, r_j), &  (m_j \ge q_j),\\[2pt]
(l_j,\, q_j - 1), & \text{otherwise.}
\end{cases}
\]
Each update keeps $m_j\in[l_j,r_j]$ and halves $r_j - l_j$, so after $O(\log n)$ rounds every bracket
collapses to $l_j = r_j = m_j$.

\emph{Cost.} In every round $q_j \le r_j \le 2 m_j$, so by \eqref{eq:mult-sum},
$q_1+\dots+q_k \le 2(m_1+\dots+m_k) \le 2n$. Each of the $O(\log n)$ rounds thus costs
$\tilde{O}(n)$ field operations.

\medskip
In total the algorithm performs $\tilde{O}(n)$ operations in $\mathbb{F}_p$. Each costs $\tilde{O}(\log p)$
time under fast multiplication, giving overall running time $\tilde{O}(n \log p)$.
\end{proof}

We now combine the previous two lemmas to find all roots of $f$ and their
multiplicities. To improve the running time, we first compute
$h:=\gcd(f,x^p-x)$ and find the roots of $h$ rather than those of $f$.
The case in which $f$ has only one root is handled separately, using the fact that $\deg h=1$ in this case.

\begin{lemma}\label{lem:find_roots_and_multiplicities}
    Let $p$ be a prime, and let $f \in \mathbb{F}_p[x]$ be a nonzero polynomial of degree $n$.
    Let $d$ be the number of distinct roots of $f$ in $\mathbb{F}_p$. Then those roots
    $x_1,\dots,x_d$ and their multiplicities $m_1,\dots,m_d$ can be computed in
    \[
        \tilde{O}\!\big(\sqrt{dp}+n\log^2 p\big)
    \]
    time. Furthermore, if $d=1$, this can be done in $\tilde{O}(n\log^2p)$ time.
\end{lemma}

\begin{proof}
Recall that each arithmetic operation in $\mathbb{F}_p$ costs
$\tilde{O}(\log p)$ time.

Since $x^p - x = \prod_{a \in \mathbb{F}_p}(x-a)$, the polynomial
\[
h := \gcd\!\big(f,\; x^p - x\big)
\]
satisfies $h = \prod_{i=1}^{d}(x - x_i)$. In particular $\deg h = d$. We compute $x^p \bmod f$ by repeated squaring
($O(\log p)$ multiplications in $\mathbb{F}_p[x]/(f)$, each $\tilde{O}(n)$ field operations),
and then $h = \gcd\!\big(f,\,(x^p \bmod f) - x\big)$ via \tool{comp:gcd} in $\tilde{O}(n)$
field operations. This costs $\tilde{O}(n\log p)$ field
operations, i.e.\ $\tilde{O}(n \log^2 p)$ time.

If $d=1$, $h$ is monic and linear, so the (only) root of $f$ is minus the constant coefficient of $h$, and is thus found in $O(1)$. Otherwise, we find the roots of $h$. $\deg h = d$, so by
\cref{comp:roots_mod_p} we recover $x_1,\dots,x_d$ in $\tilde{O}(d+\sqrt{dp})$ time.

Next, we compute their multiplicities in $f$.
Applying \cref{lem:find_multiplicities} to $f$ and the roots $x_1,\dots,x_d$ yields
$m_1,\dots,m_d$ in $\tilde{O}(n\log p)$ time.

Summing the running times of the steps, the algorithm runs in $\tilde{O}(n\log^2p+\sqrt{dp})$ time for $d \neq 1$, and $\tilde{O}(n\log^2 p)$ for $d=1$.

\end{proof}

\section{Finding a good prime $p$}\label{sec:finding-p}

We first prove that a prime number $p$ satisfying all three requirements of \cref{thm:finding_p} exists; we then give an algorithm for finding it.

The most challenging property that $p$ must satisfy is having $\deg(\gcd(f_p,f_p')) < t$. For that purpose, we introduce the resultant, and prove that $\deg (\gcd(f_p,f_p'))\le v_p(\operatorname{res}(f,f'))$. Thus, by bounding $|\operatorname{res}(f,f')|$, we bound the number of ``bad'' primes satisfying $\deg(\gcd(f_p,f_p')) \ge t$. Removing those bad primes, and the primes that divide the leading coefficient of $f$ from the list of primes smaller than $\tilde{O}(\frac{nb}{t})$ gives a prime satisfying all three requirements. 

We begin by defining the resultant and proving the desired results regarding it.
\subsection{The Resultant}
Let \(R\) be an integral domain, and let
\[
A(x)=a_nx^n+a_{n-1}x^{n-1}+\cdots+a_1x+a_0
\]
and
\[
B(x)=b_mx^m+b_{m-1}x^{m-1}+\cdots+b_1x+b_0
\]
be polynomials in \(R[x]\) of degrees \(n\) and \(m\), respectively.

The resultant of \(A\) and \(B\), denoted \(\operatorname{res}(A,B)\), is the
determinant of their Sylvester matrix:
\[
\operatorname{res}(A,B)
=
\det
\begin{pmatrix}
a_n & a_{n-1} & \cdots & a_0 & 0 & \cdots & 0 \\
0 & a_n & a_{n-1} & \cdots & a_0 & \cdots & 0 \\
\vdots & & \ddots & & & \ddots & \vdots \\
0 & \cdots & 0 & a_n & a_{n-1} & \cdots & a_0 \\
b_m & b_{m-1} & \cdots & b_0 & 0 & \cdots & 0 \\
0 & b_m & b_{m-1} & \cdots & b_0 & \cdots & 0 \\
\vdots & & \ddots & & & \ddots & \vdots \\
0 & \cdots & 0 & b_m & b_{m-1} & \cdots & b_0
\end{pmatrix}.
\]
This is an $(n+m)\times(n+m)$ matrix: its first $m$ rows contain shifted copies of the coefficients of $A$, and its last $n$ rows contain shifted copies of the coefficients of $B$. (For simplicity, the matrix is drawn here for $m=n$.)

Equivalently, let \(R[x]_{<d}\) denote the \(R\)-module of polynomials of
degree less than \(d\). The above Sylvester matrix represents the
\(R\)-linear map
\[
T : R[x]_{<m} \oplus R[x]_{<n} \longrightarrow R[x]_{<n+m}
\]
defined by
\[
T(P,Q)=AP+BQ.
\]
With respect to the standard coefficient bases, this matrix is the
Sylvester matrix. Hence
\[
\operatorname{res}(A,B)=\det(T).
\]

An important property of the resultant is that $\operatorname{res}(A,B) = 0$ if and only if $A$ and $B$ have a common non-constant divisor over the fraction field of $R$.

We now prove two helper lemmas from which \cref{lem:res_modular_bound} will follow.

\begin{lemma}\label{lem:res_matrix_rank_bound}
    Let $p$ be a prime number and let $A,B \in {\mathbb{F}_p[x]}$ be polynomials with $\deg A \le n$, $\deg B \le m$. Denote $g = \gcd(A,B)$ and $d = \deg(g)$. Then the rank of the Sylvester matrix of $A,B$ is bounded by $n+m-d$.
\end{lemma}
\begin{proof}
     The Sylvester matrix of $A,B$ represents the linear transformation $T(P,Q) = AP+BQ$. The image of $T$ is contained in the vector space consisting of polynomials of degree less than $n+m$ which are divisible by $g$, which is of dimension $n+m-d$.
\end{proof}

\begin{lemma}\label{lem:v_p_bound_from_rank}
    Let $p$ be a prime number. Let $M \in M_n(\mathbb{Z})$ be a matrix, and let $M_p \in M_n(\mathbb{F}_p)$ denote the reduction of $M$ modulo $p$. If the rank of $M_p$ over $\mathbb{F}_p$ is at most $n-d$, then $p^d \mid \det(M)$.
\end{lemma}

\begin{proof}
    We perform a Gaussian elimination on $M_p$ over $\mathbb{F}_p$ to obtain an invertible matrix $X_p \in GL_n(\mathbb{F}_p)$ such that $M_pX_p$ contains $d$ zero columns. Now, choose some matrix $X \in M_n(\mathbb{Z})$ whose reduction modulo $p$ is $X_p$. Note that the matrix $MX$ has $d$ columns which vanish modulo $p$. Thus, by expansion of permutations, $p^d \mid \det(MX)$. Since $\det(M) = \frac{\det(MX)}{\det(X)}$ and $p \nmid \det(X)$, we have $p^d \mid \det(M)$. 
\end{proof}

\begin{lemma}\label{lem:res_modular_bound}
    Let $p$ be a prime number and let $A,B \in \mathbb{Z}[x]$. Denote $g = \gcd(A_p,B_p)$ and $d = \deg(g)$. Then $p^d \mid \operatorname{res}(A,B)$.
\end{lemma}
\begin{proof}
    Follows immediately by applying \cref{lem:res_matrix_rank_bound} to $A_p,B_p$ and \cref{lem:v_p_bound_from_rank} to the Sylvester matrix.
\end{proof}

\begin{lemma}\label{lem:res_size_bound}
    For $A,B \in \Z[x]$ with $\deg A = n$, $\deg B = m$, 
    \[|\operatorname{res}(A,B)| \le \max(\|A\|_\infty,\|B\|_\infty)^{n+m}(n+m)!\]
\end{lemma}
\begin{proof}
    Note that the Sylvester matrix $M$ is an $(n+m)\times(n+m)$ matrix with entries that are bounded (in absolute value) by $\max(\|A\|_\infty,\|B\|_\infty)$. Thus, every summand in the expansion by permutations of $\det(M)$ is bounded (in absolute value) by $\max(\|A\|_\infty,\|B\|_\infty)^{n+m}$. Therefore, $|\operatorname{res}(A,B)| = |\det(M)| \le \max(\|A\|_\infty,\|B\|_\infty)^{n+m}(n+m)!$
\end{proof}

\subsection{Proof of existence}

We bound the number of primes that violate requirement \ref{req:multiplicities_are_small}.

\begin{lemma}\label{lem:not_lots_of_bad_primes}
    Let $n,b,t$ be positive integers and $f \in \mathbb{Z}[x]$ be a square-free polynomial of degree $n$ with $\|f\|_\infty < 2^b$. Then there are at most $r := \lceil\frac{6nb+6n\log(n)}{t}\rceil$ primes $p$ that satisfy $\deg(\gcd(f_p,f_p'))\ge t$.
\end{lemma}

\begin{proof}
    Let $a$ be the number of primes $p$ with $\deg(\gcd(f_p,f_p'))\ge t$, and let $M$ be their product. By \cref{lem:res_modular_bound}, every such $p$ satisfies $p^{t} \mid \operatorname{res}(f,f')$, and hence $M^{t} \mid \operatorname{res}(f,f')$. Since $f$ is square-free, $\operatorname{res}(f,f')\neq 0$, and therefore
    \[
        2^{ta} \le M^t \le |\operatorname{res}(f,f')|.
    \]

    For an upper bound, we apply \cref{lem:res_size_bound} to $A=f$ and $B=f'$, which have degrees $n$ and $n-1$, and satisfy $\|f\|_\infty < 2^b$ and $\|f'\|_\infty \le n2^b$. Thus
    \[
        |\operatorname{res}(f,f')| \le (n2^b)^{2n-1}(2n-1)! \le (n2^b)^{2n}(2n)^{2n} = 2^{2nb+4n\log(n)+2n} \le 2^{6nb+6n\log(n)},
    \]
    where the last inequality uses $b\ge 1$.

    Combining the two bounds gives $2^{ta}\le 2^{6nb+6n\log(n)}$, so $a \le \frac{6nb+6n\log(n)}{t}$. Since $a$ is an integer, $a\le r$.
\end{proof}

As the leading coefficient of $f$ is bounded by $2^b$, the number of primes violating requirement \ref{req:pndiv_a_n} is at most $b$. 

Let $u = b+r+1$. Among the first $u$ prime numbers, at most $b$ violate requirement \ref{req:pndiv_a_n}, and at most $r$ of the remaining ones violate requirement \ref{req:multiplicities_are_small}. Hence there must be a prime $p$ among them satisfying requirements \ref{req:multiplicities_are_small} and \ref{req:pndiv_a_n}. Moreover, $p$ is at most the $u$-th prime number, which is $\tilde{O}(u)=\tilde{O}(\frac{nb}{t})$ by the prime number theorem (using $t\le n$). Thus $p$ also satisfies requirement \ref{req:p_is_small}.
 
\subsection{Algorithm}

Our strategy for finding the appropriate prime is to iterate over the prime numbers in increasing order until we find one that satisfies all three requirements. To avoid reading all coefficients of $f$ each time we iterate over some prime, we reduce each coefficient of $f$ modulo each prime in advance.

\begin{algorithm}[H]
    \caption{Finding a good prime $p$}
    \label{alg:algorithm_find_p}
    \raggedright
    Input: Integers $n,b$, a parameter $1 \le t \le n$, and a square-free polynomial $f(x) = a_nx^n + \ldots +a_1x+a_0$ of degree $n$ with $\|f\|_\infty < 2^b$.  \\
    Output: a prime $p$ satisfying the three requirements of \cref{thm:finding_p}.
    \noindent\rule{\textwidth}{0.4pt}
    \begin{algorithmic}[1]
        \State Set $r = \lceil\frac{6nb+6n\log(n)}{t}\rceil$, \(u =  b+r+1\) \label{step:first-line}
        \State Compute an array $P$ consisting of the first $u$ prime numbers using \tool{comp:sieve} \label{step:sieve}
        \State Compute the residue of $a_n$ modulo every prime in $P$ using \tool{comp:multipoint_modulo} \label{step:leading_coeff_residues}
        \State Compute an array $Q$ consisting of the first $r+1$ primes in $P$ that do not divide $a_n$ \label{step:calc_Q}
        \For{$i = 0,1,\ldots,n$}\label{step:first loop}
            \State Compute the remainder of $a_i$ modulo every $p$ in $Q$ using \tool{comp:multipoint_modulo} \label{step:coeffs_residues}
        \EndFor
        \For{every $p$ in $Q$}\label{step:second_loop}
            \State Differentiate $f_p$ to get the coefficients of $f_p'$ \label{step:calc_fp_prime}
            \State Compute $g = \gcd(f_p,f_p')$ using \tool{comp:gcd} \label{step:compute gcd}
            \If{$\deg(g) < t$}
                \State Output $p$ and terminate \label{step:found_p}
            \EndIf
        \EndFor
    \end{algorithmic}
\end{algorithm}

\subsubsection{Correctness}
Every prime in $P$ is one of the first $u$ prime numbers. By the prime number theorem
the $u$-th prime is $\tilde{O}(u) = \tilde{O}(\frac{nb}{t})$, so every prime in $P$
satisfies requirement \ref{req:p_is_small}. By construction no prime in $Q$ divides $a_n$,
so every prime in $Q$ satisfies requirement \ref{req:pndiv_a_n}.

We first check that $Q$ is well defined, i.e.\ that $P$ contains at least $r+1$ primes
that do not divide $a_n$. Since $|a_n| < 2^b$, at most $b$ primes divide $a_n$. As $P$
contains $u = b+r+1$ primes, at least $r+1$ of them do not divide $a_n$, so $Q$ indeed
contains $r+1$ primes.

It remains to show that some prime in $Q$ satisfies requirement
\ref{req:multiplicities_are_small}. By \cref{lem:not_lots_of_bad_primes}, the number of primes violating
requirement \ref{req:multiplicities_are_small} is at most $r$.
Since $Q$ contains $r+1$ primes, at least one of them satisfies requirement
\ref{req:multiplicities_are_small}. The loop of line \ref{step:second_loop} returns the
first such prime at line \ref{step:found_p}, which therefore satisfies all three requirements.

\subsubsection{Time analysis}
Recall $r = \lceil\frac{6nb+6n\log(n)}{t}\rceil = \tilde{O}(\frac{nb}{t})$ and $u = b+r+1 = \tilde{O}(\frac{nb}{t})$,
the latter since $t \le n$. We show that the algorithm runs in $\tilde{O}(u+nr) = \tilde{O}(\frac{n^2 b}{t})$ time.

By the prime number theorem every prime in $P$ is $\tilde{O}(u)$, so each arithmetic
operation modulo such a prime costs $\tilde{O}(\log(nb))$ time.
This factor is polylogarithmic in the input and is absorbed into the $\tilde{O}$ of the final bound.

Computing the first $u$ primes (line \ref{step:sieve}) costs $\tilde{O}(u)$ by \tool{comp:sieve}.
Reducing $a_n$ modulo the at most $u$ primes of $P$ (line \ref{step:leading_coeff_residues}) costs
$\tilde{O}(b+u) = \tilde{O}(u)$ by \tool{comp:multipoint_modulo}, and selecting $Q$ (line \ref{step:calc_Q})
is a linear scan over $P$, hence $\tilde{O}(u)$.

The two loops dominate. In lines \ref{step:first loop}--\ref{step:coeffs_residues}, each of the $n+1$
coefficients is reduced modulo the $r+1$ primes of $Q$, costing $\tilde{O}(b+r)$ per coefficient by
\tool{comp:multipoint_modulo}, for $\tilde{O}(n(b+r)) = \tilde{O}(nr)$ in total (using $r \ge \frac{nb}{t} \ge b$).
In lines \ref{step:second_loop}--\ref{step:found_p}, each of the at most $r+1$ primes requires
differentiating $f_p$ and one gcd computation (\tool{comp:gcd}), i.e.\ $\tilde{O}(n)$ field operations,
for $\tilde{O}(nr)$ in total.

Altogether the algorithm runs in $\tilde{O}(u + nr) = \tilde{O}(\frac{n^2 b}{t})$ time.

\section{Lifting roots}\label{sec:lifting}

\subsection{Background and general ideas}

The following lifting problem is classical. Let $f \in \Z[x]$ and let $p$ be a prime. Given an integer $x_0$ with $p \mid f(x_0)$ and $p \nmid f'(x_0)$, together with a positive integer $m$, find an integer $x$ such that $x \equiv x_0 \pmod p$ and $p^m \mid f(x)$. Such an $x$ always exists, is unique modulo $p^m$, and can be computed efficiently; these are exactly the assertions of Hensel's lifting lemma.

This lemma yields a classic algorithm for finding the non-negative integer roots of a polynomial; negative roots can be handled by applying the same
procedure to $f(-x)$. First, apply \cref{thm:finding_p} with $t=1$ to obtain a prime $p$ for which $f_p$ has no repeated roots. Next, solve $f(x) \equiv 0 \pmod p$ and use Hensel's lemma to lift each solution to a root modulo $p^b$. Every integer root of $f$ divides its constant term and is therefore bounded by $2^b$; since the lift is unique modulo $p^b$, every such root coincides with one of these lifts. As $f$ has at most $n$ roots modulo $p$, this produces a list of at most $n$ candidates, which we pass to \cref{thm:verifying_roots} to recover the roots of $f$. 

The difficulty is that running \cref{thm:finding_p} with $t=1$ costs $\tilde{O}(n^2 b)$ time, which exceeds our target. We instead run \cref{thm:finding_p} with $t=\sqrt{n}$ and show how to lift repeated roots efficiently. Before turning to the algorithm, we develop the mathematics underlying the lifting of repeated roots.

Lifting a repeated root is delicate because such a root may admit many lifts. Consider $f(x)=x^2$ with $p=3$, and the root $x_0=0$ of $x^2 \equiv 0 \pmod 3$. It has three lifts modulo $9$, namely $0,3,6$; more generally, the number of solutions of $x^2 \equiv 0 \pmod{3^m}$ equals $3^{\lfloor m/2 \rfloor}$, which grows exponentially in $m$. Thus, we need a different notion of lifting. To this end we introduce the lifting tree.

\subsection{The lifting tree}
Let us fix a prime number $p$ and a polynomial $f \in \Z[x]$ of degree $n$. We fix some notations.

\begin{definition}
For a non-zero polynomial $h \in \Z[x]$, we write $v(h) = v_p(h)$ and define $h^* = \frac{h}{p^{v(h)}}$.   
We denote by $h_p^*$ the image of $h^*$ in $\mathbb{F}_p[x]$.
\end{definition}

\begin{observation}
We note that all coefficients of $h^*$ are integers, and at least one of them is not divisible by $p$, which implies that  $h_p^*$ is always non-zero.    
\end{observation}

\begin{definition}
For a polynomial $f \in \Z[x]$ and an integer $k \ge 0$, the
\emph{$k$-th Hasse derivative} of $f$ is
\[
    f^{[k]} := \frac{1}{k!}\, f^{(k)}.
\]
Equivalently, the Hasse derivatives are determined by the identity
\[
    f(X + Y) = \sum_{k \ge 0} f^{[k]}(X)\, Y^k
\]
in $\Z[X,Y]$.
\end{definition}

\begin{observation}\label{obs:hasse}
Writing $f(X) = \sum_j a_j X^j$, we have
$f^{[k]}(X) = \sum_j \binom{j}{k} a_j\, X^{j-k}$, so every coefficient of $f^{[k]}$ lies in $\Z$ and is also divisible by $p^{v(f)}$; in particular $v_p\big(f^{[k]}(c)\big) \ge v(f)$ for every
integer $c$.
\end{observation}

\begin{lemma}\label{lem:root_multiplicity_in_F_p}
Let $f \in \Z[X]$ and let $c\in\Z$ be such that
$v_p\big(f^{[k]}(c)\big) > v(f)$ for all $0 \le k < m$. Then $c \bmod p$ is a root of
multiplicity at least $m$ of $f_p^*$.
\end{lemma}

\begin{proof}
Dividing $f$ by $p^{v(f)}$ we may assume $v(f) = 0$, so that $f_p^* = f_p$ and the
hypothesis becomes $p \mid f^{[k]}(c)$ for $0 \le k < m$. The expansion
$f(c + Y) = \sum_{k \ge 0} f^{[k]}(c)\, Y^k$ in $\Z[Y]$ gives that the first $m$ coefficients of $f_p(c+Y)\in\mathbb F_p[Y]$ vanish. Thus $c \bmod p$ is a root of multiplicity at least $m$ of $f_p=f_p^*$.
\end{proof}

In this section, integers are represented in base $p$. Each integer has some number of digits in base $p$ which will be clear from the context. For example, the integers $1$ and $01$ in base $p$ do not represent the same integer because they do not have the same number of base $p$ digits, and therefore their lifts may result in different sets of numbers. For an integer $x$ with $m$ digits, we set 
\[f_x(a) := f(x+p^ma),\]
where $a$ is a new variable, which represents the ``continuation'' of the base-$p$ expansion of $x$ that we are looking for. $f_x$ is called the branching polynomial of $x$. We think of the nodes of the lifting tree as integers approximating the $p$-adic roots of $f$.

The root of the lifting tree is a $0$-digit integer. Its value is equal to $0$, but we denote it as the empty string $\epsilon$ in order to distinguish it from the $1$-digit integer $0$. Note that $f_\epsilon=f$.

Consider a node \(x_0\) in layer \(m\). We seek roots whose first \(m\)
base-\(p\) digits agree with \(x_0\). Such a root has the form
\[
    x=x_0+p^m \beta
\]
for some $\beta\in\mathbb Z_p$. Therefore it must satisfy
\[
    f_{x_0}(\beta)=f(x)=0.
\]
After removing the common power \(p^{v(f_{x_0})}\) from the coefficients, we see that the
next digit \(\delta=\beta\bmod p\) must be a root of \((f_{x_0})_p^*\) in \(\mathbb F_p\), that is,
\[
    (f_{x_0})_p^*(\delta)=0.
\] For every such digit \(\delta\), we add the
\((m+1)\)-digit integer \(x_0+p^m\delta\) as a child of \(x_0\), and label the
edge by \(\delta\).

We emphasize that for efficiency purposes, the algorithm will not store a node by its value as an
integer. Instead, it stores the lifting tree as a rooted tree whose edges are
labeled by digits in \(\{0,\ldots,p-1\}\). A node is identified with the path
from the root to that node. If the edge labels on such a path are
\(\delta_0,\delta_1,\ldots,\delta_{m-1}\), then the corresponding \(m\)-digit
integer is
\[
    \delta_0+\delta_1p+\cdots+\delta_{m-1}p^{m-1}.
\]
This interpretation is used only when we finally turn layer-\(m\) nodes into
integer candidates. All intermediate computations use only the edge labels on
paths, not the integer values of the nodes.

Normalizing $f_{x_0}$ by $p^{v(f_{x_0})}$ before
reducing modulo $p$ is important. For example, take $f(x) = x^2$, and $p = 3$, and lift the $1$-digit node
$0$. Its branching polynomial is $f_0(a) = f(3a) = 9a^2$, so $v(f_0) = 2$ and
$(f_0)_3^*(a) = a^2$. The equation $a^2 = 0$ has the single solution
$a \equiv 0 \pmod 3$, producing only the $2$-digit child $00$. Had we instead
reduced $f_0$ modulo $p$ without normalizing, we would obtain the zero
polynomial, which every residue $a \in \{0,1,2\}$ solves; this would produce the extra children $01, 02$.

We call the $m$-digit integers in the tree the $m$-th layer of the tree.

Observe that the children of $\epsilon$ (the root of the tree) are the solutions of $f(x) \equiv 0 \pmod p$; and that the lifting tree might be infinite.

We show that every root of $f$ in $\mathbb Z_p$ determines an infinite
path in the lifting tree starting at $\epsilon$. 

\begin{lemma}\label{lem:lifting_tree_contains_roots}
Let $\alpha\in \mathbb Z_p$ be a root of $f$. Write
$$
\alpha=\sum_{i=0}^{\infty} \alpha_i p^i,
\qquad \alpha_i\in\{0,\ldots,p-1\}.
$$
Then the digits $\alpha_0,\alpha_1,\ldots$ form an infinite path in the lifting
tree starting at $\epsilon$. Equivalently, for every $m\geq 0$, the vertex
corresponding to the prefix
$$
\alpha^{(m)}:=\sum_{i=0}^{m-1}\alpha_i p^i
$$
appears in layer $m$ of the lifting tree.
\end{lemma}

\begin{proof}
We prove the claim by induction on $m$. The case $m=0$ is the root $\epsilon$ of
the tree. Assume that the prefix
$$
\alpha^{(m)}=\sum_{i=0}^{m-1}\alpha_i p^i
$$
appears in layer $m$. Since $\alpha$ and $\alpha^{(m)}$ have the same first $m$
base-$p$ digits, we may write
$$
\alpha=\alpha^{(m)}+p^m\beta
$$
for some $\beta\in\mathbb Z_p$. Moreover,
$$
\beta \equiv \alpha_m \pmod p.
$$
By the definition of the shifted polynomial attached to the node
$\alpha^{(m)}$, we have
$$
f_{\alpha^{(m)}}(\beta)=f(\alpha)=0.
$$
Hence the reduction of $\beta$ modulo $p$, namely $\alpha_m$, is a root of
$(f_{\alpha^{(m)}})_p^*$. Therefore the lifting tree contains a child of
$\alpha^{(m)}$ whose next digit is $\alpha_m$. This child is exactly the prefix
$$
\alpha^{(m+1)}=\alpha^{(m)}+\alpha_m p^m.
$$
Thus the prefix of length $m+1$ appears in layer $m+1$. By induction, all
prefixes of $\alpha$ appear in the tree, so the digits of $\alpha$ determine an
infinite path starting at $\epsilon$.
\end{proof}

Our objective in this section will be to compute layer $b$ of the lifting tree. Since each non-negative integer root $x$ of $f$ satisfies $x < 2^b \le p^b$, it has at most $b$ digits in base $p$, so the lemma above proves that $x$ appears in layer $b$ of the lifting tree.

To bound the amount of time it takes us to compute the lifting tree, we need a quantity that controls the number of nodes in every layer. The relevant quantity is
the following weight function.

\begin{definition}
Let $x \neq \epsilon$ be a node with parent $y$, reached from $y$ by the digit $\delta$.
The \emph{weight} $w(x)$ is the multiplicity of $\delta$ as a root of $(f_y)_p^*$.
For the root we set $w(\epsilon) := \deg (f_\epsilon)_p^* = \deg f_p^*$.
\end{definition}

\begin{lemma}\label{lem:deg_le_weight}
For every node $x$, $\deg (f_x)_p^* \le w(x)$.
\end{lemma}
\begin{proof}
For the root the claim holds with equality, since $w(\epsilon) = \deg (f_\epsilon)_p^*$
by definition. So let $x \neq \epsilon$ be a node with parent $y$ at layer $m$, reached from
$y$ by the digit $\delta$; thus $x = y + p^m \delta$ with $0 \le \delta \le p-1$, and
$w(x)$ is the multiplicity of $\delta$ as a root of $(f_y)_p^*$. Write $D := \deg (f_x)_p^*$.

Expanding $f_y$ around $\delta$,
\[
    f_x(a) = f_y(\delta + pa) = \sum_{k=0}^{n} f_y^{[k]}(\delta) \, p^{k} \, a^{k},
\]
so the coefficient of $a^k$ in $f_x$ is $f_y^{[k]}(\delta)\,p^k$. By definition $D$ is the
largest index whose coefficient attains the minimal valuation $v(f_x)$ among the
coefficients of $f_x$. Hence, using $p^{v(f_y)} \mid f_y^{[D]}(\delta)$ from
\cref{obs:hasse},
\[
    v(f_x) = v_p\!\big(f_y^{[D]}(\delta)\,p^{D}\big)
           = v_p\!\big(f_y^{[D]}(\delta)\big) + D
           \;\ge\; v(f_y) + D.
\]

Now fix $0 \le k < D$. Since every coefficient of $f_x$ has valuation at least $v(f_x)$,
\[
    v_p\!\big(f_y^{[k]}(\delta)\big) + k = v_p\!\big(f_y^{[k]}(\delta)\,p^k\big)
        \ge v(f_x) \ge v(f_y) + D,
\]
so $v_p\!\big(f_y^{[k]}(\delta)\big) \ge v(f_y) + (D - k) \ge v(f_y) + 1 > v(f_y)$. By
\cref{lem:root_multiplicity_in_F_p} applied to $f_y$ at $\delta$, the digit $\delta$ is a
root of $(f_y)_p^*$ of multiplicity at least $D$. That multiplicity is exactly $w(x)$, so
$\deg (f_x)_p^* = D \le w(x)$, as desired.
\end{proof}

\begin{lemma}\label{lem:weights_of_children_sum_bound}
    The sum of weights of children of a node $x$ is at most $w(x)$.
\end{lemma}
\begin{proof}
Each child of $x$ has a distinct last digit, which is a root of $(f_x)_p^*$, and its
weight is the multiplicity of that digit in $(f_x)_p^*$. Multiplicities of distinct
roots sum to at most the degree, so the children's weights sum to at most
$\deg(f_x)_p^*$, which is at most $w(x)$ by \cref{lem:deg_le_weight}.
\end{proof}
\begin{lemma}\label{lem:tree_layers_have_atmost_n_nodes}
    Every layer of the lifting tree contains at most $w(\epsilon)=\deg(f_p^*)$ nodes.
\end{lemma}
\begin{proof}
By an easy induction using the lemma above, the sum of weights of the nodes in any layer of the tree is bounded by $\deg(f_p^*)$. Since $w(x) \ge 1$ for every node $x$, the number of nodes in a given layer is bounded by the sum of their weights, which is at most $\deg(f_p^*)$.
\end{proof}

The previous lemmas show that the lifting tree has few nodes in each layer.
For an efficient algorithm, this is not enough: the branching polynomials
\(f_x\) themselves may have very large coefficients. Fortunately, to find the
children of \(x\), we only need \((f_x)_p^*\). This can be recovered from
\(f_x\) modulo \(p^{v(f_x)+1}\). Hence we need an upper bound on \(v(f_x)\)
along the tree.

\begin{lemma}\label{lem:tree_equations_dont_have_large_vp}
Let $x$ be a node in layer $m$, and let $x_0$ denote the residue of $x$ mod $p$. Denote $\mu = w(x_0)$. Then $v(f_x)\le v(f)+\mu m$.
\end{lemma}

\begin{proof}
    If $m=0$, the bound is trivial. Dividing $f$ by $p^{v(f)}$, we may assume $v(f)=0$. The coefficient of $a^\mu$ in $f_x(a)$ is $f^{[\mu]}(x)p^{m\mu}$. Since $v(f)=0$, the weight $\mu=w(x_0)$ is the multiplicity of $x_0$ as a root of $f_p$. Reducing the identity $f(x_0+Y)=\sum_{k\ge 0}f^{[k]}(x_0)Y^k$ modulo $p$, we see that the coefficients of $f_p(x_0+Y)$ are the residues of $f^{[k]}(x_0)$ modulo $p$, and the lowest nonzero one is the coefficient of $Y^\mu$. Hence $p \nmid f^{[\mu]}(x_0)$. Thus $p \nmid f^{[\mu]}(x)$ as $x \equiv x_0 \pmod p$. As a corollary, $v(f_x) \le v_p(f^{[\mu]}(x)p^{m\mu}) = m\mu$.
\end{proof}

\begin{lemma}\label{lem:coefficients_matter_only_mod_p^sb}
    Suppose that for every layer-1 node $x_0$ we have $w(x_0)\le \mu$. Then layers 0 to $m+1$ of the lifting tree depend only on the coefficients of $f$ modulo $p^{v(f)+m\mu+1}$.
\end{lemma}

\begin{proof}
    We prove this by induction on the layer number, with a trivial base case. Suppose the claim is proved for some layer $k\le m$. The children of a node $x$ on layer $k$ depend only on $x$ and $(f_x)_p^*$. Let $x_0 = x \pmod p$. Then by the previous lemma, $v(f_x) \le  v(f) + w(x_0)k \le v(f)+\mu k \le v(f)+\mu m$, so the coefficients of $(f_x)_p^*$ depend only on $x$ and on the coefficients of $f$ modulo $p^{v(f)+\mu m+1}$. Thus the claim is proved. 
\end{proof}

The next lemma will only be used in \cref{sec:verifying}. We place it here because the proof of the lemma requires tools from this section.

\begin{lemma} \label{x_appears_on_lifting_tree}
    Let $a_0\neq 0$ denote the constant coefficient of $f$, and let $b>1$ be an integer. Suppose that $p \mid a_0$. Let $k = v_p(a_0)$. Then there are at most $k$ distinct non-negative integers $x$ satisfying $x < p^b$, $p \mid x$, and $p^{bk} \mid f(x)$.
\end{lemma}

\begin{proof}
    Let $g(y) = f(py)$. Then $v(g) \le v_p(a_0) = k$, while every coefficient of $g$
of degree greater than $k$ is divisible by $p^{k+1}$: the degree-$i$ coefficient
of $g$ is $a_i p^i$, of valuation $v_p(a_i) + i \ge i > k$ for $i > k$.
Consequently $\deg g_p^* \le k$.
 
Write $x = pz$. We show by induction on $m$ that for every $0 \le m \le b-1$, the
first $m$ digits of $z$ form a node on layer $m$ of the lifting tree of $g$. The
base case $m = 0$ is trivial.
 
Assume the first $m$ digits of $z$ give a node $z_0$ on layer $m$, with
$m \le b-2$, and write $z = z_0 + p^ma$. By \cref{lem:tree_equations_dont_have_large_vp},
\[
    v(g_{z_0}) \le v(g) + m \deg(g_p^*) \le k + mk \le k(b-1).
\]
On the other hand,
\[
    g_{z_0}(a) = g(z_0 + p^m a) = g(z) = f(pz) = f(x),
\]
so by hypothesis $v_p\big(g_{z_0}(a)\big) \ge kb$. Hence
\[
    v_p\big(g_{z_0}^*(a)\big)
    = v_p\big(g_{z_0}(a)\big) - v(g_{z_0})
    \ge kb - k(b-1) = k \ge 1,
\]
so $a \bmod p$ is a root of $(g_{z_0})_p^*$, and the corresponding child of $z_0$
is exactly the node given by the first $m+1$ digits of $z$. This completes the
induction.

Combining $\deg g_p^* \le k$ with
\cref{lem:tree_layers_have_atmost_n_nodes} gives that layer $b-1$ of the lifting tree
of $g$ contains at most $k$ nodes. As $z = \frac{x}{p} < p^{b-1}$, $z$ has $b-1$ digits, so every admissible $x$ yields a distinct $z$ on that
layer. Thus there are at most $k$ such values of $x$.
\end{proof}

\subsection{Tree building algorithm}

Throughout this section, the $\tilde{O}$ notation hides $\log p$ factors that come from arithmetic operations on $p$-adic digits. This is consistent with the complexity requirement of \cref{thm:lifting_roots} since $\log p = O(\log(nbt+(\sqrt{n}+t)\sqrt{p}))$.  

We now turn the preceding structural lemmas into an algorithm. The difficulty is that the branching polynomials \(f_x\) may have
coefficients with many \(p\)-adic digits. Therefore, every polynomial is maintained only modulo a
suitable power of \(p\).

For a polynomial \(g\in\mathbb Z[x]\), saying that \(g\) is known to precision
\(k\) means that we
know \(g\bmod p^k\). That is, each coefficient of \(g\) is known modulo \(p^k\).

The precision is chosen so that it is large enough to
determine relevant future branching decisions, but small enough that arithmetic
remains efficient.

In particular, in order to compute the children of some node $x$, $f_x^*$ must be known to precision at least $1$.

The basic operation in the algorithm is the following. Suppose that we already
know the normalized branching polynomial $f_{x_1}^*$ at some node \(x_1\), and we know
the digits along the path from \(x_1\) to a descendant \(x_2\). Then we can
recover the normalized branching polynomial at \(x_2\), losing only a bounded amount
of precision.

\begin{lemma}\label{lem:computing_branching_polys}
    Let $x_1,x_2$ be nodes of layers $m_1<m_2$ such that $x_1$ is an ancestor of $x_2$. We are given $f_{x_1}^*$ to precision $k$, and the list of digits along the path from $x_1$ to $x_2$. Let $y$ be the child of $x_1$ on the path from $x_1$ to $x_2$. Then, assuming $k \ge w(y)(m_2-m_1),$ we can compute $f_{x_2}^*$ to precision $k-w(y)(m_2-m_1)$ in $\tilde{O}(nk)$ time. 
\end{lemma}

\begin{proof}
    Writing $x_2 = x_1 + p^{m_1}a$, we see that the digits on the path from $x_1$ to $x_2$ form the integer $a$ in base $p$. By the definition of branching polynomials,
    \[
    f_{x_2}(z) = f(x_1+p^{m_1}a+p^{m_2}z) = f_{x_1}(a+p^{m_2-m_1}z).
    \]
    Define
    \[
    f_{x_2}^{\diamondsuit}(z) = f_{x_1}^*(a+p^{m_2-m_1}z).
    \]
    Since $f_{x_1}^* = f_{x_1}/p^{v(f_{x_1})}$, we have $f_{x_2}^{\diamondsuit} = f_{x_2}/p^{v(f_{x_1})}$, and hence $(f_{x_2}^{\diamondsuit})^* = f_{x_2}^*$.

    We use the following simple fact: if we know a polynomial $g$ to precision $k$ and $v(g)\le k$, then we know $g^*=\frac{g}{p^{v(g)}}$ to precision $k-v(g)$.

    \paragraph{The algorithm.} Since $f_{x_2}^{\diamondsuit}$ is a Taylor shift of $f_{x_1}^*$, we compute $f_{x_2}^{\diamondsuit} \bmod p^k$ from $f_{x_1}^* \bmod p^k$ using \tool{comp:taylor_shift}. Then we output $\frac{f_{x_2}^{\diamondsuit}}{p^{v(f_{x_2}^{\diamondsuit})}} \bmod p^{k-w(y)(m_2-m_1)}$.

    \paragraph{Correctness.} By the fact above with $g = f_{x_2}^{\diamondsuit}$, the output is $(f_{x_2}^{\diamondsuit})^* = f_{x_2}^*$ to precision $k-w(y)(m_2-m_1)$, provided that $v(f_{x_2}^{\diamondsuit}) \le w(y)(m_2-m_1)$. This bound follows from \cref{lem:tree_equations_dont_have_large_vp} applied to the lifting tree of $f_{x_1}^*$.

    \paragraph{Running time.} By \tool{comp:taylor_shift}, computing $f_{x_2}^{\diamondsuit} \bmod p^k$ takes $\tilde{O}(n)$ operations in $\mathbb Z/p^k\mathbb Z$, which takes $\tilde{O}(nk)$ time. Reading $a \bmod p^k$ off the list of digits, computing $v(f_{x_2}^{\diamondsuit})$, and dividing by $p^{v(f_{x_2}^{\diamondsuit})}$ only involve reading and shifting base-$p$ digits, which takes $O(nk)$ time.
\end{proof}

Our goal is to prove the following.

\begin{lemma}\label{lem:build_tree_cost}
Let $f \in \Z[x]$ be given to precision $bW$, with $v(f)=0$ and $\deg f_p \le W$. Then layers $0$ to $b$ of the lifting tree, together with all node weights, can be computed in $\tilde{O}(W(nb+\sqrt{p}))$ time.
\end{lemma}

\paragraph{Proof overview.}

The construction is by divide and conquer on the number of layers. First we
construct the first \(b/2\) layers of the lifting tree of \(f\). For each node
\(x\) in layer \(b/2\), the descendants of \(x\) are determined by
\(f_x^*\). Thus, once \(f_x^*\) is known to sufficient precision, we may
recursively construct the next \(b/2\) layers below \(x\). This gives the first \(b\) layers of the
lifting tree of \(f\).

It remains to solve the following problem: given $k=\frac{b}{2}$ layers of the lifting tree, compute the branching polynomials of the nodes of layer $k$ to an appropriate precision efficiently.

Computing all layer-\(k\) branching polynomials
independently is expensive. Instead, we decompose the known tree into one heavy path, called
the \emph{spine}, and a collection of light subtrees branching off it. Then, we compute the branching polynomial of each spine node to a sufficient precision, and use it to recurse on each light subtree.

Figures~\ref{fig:spine}, \ref{fig:tree-recursion} and \ref{fig:finishing} illustrate the flow of the algorithm. 

\begin{proof}
    
We now define the spine. Starting at $\epsilon$, repeatedly descend to a child of weight $>W/2$ while one exists.
By \cref{lem:weights_of_children_sum_bound} the children of a node have weights summing to
at most its own, so a node has at most one child of weight $>W/2$; the descent therefore
traces a well-defined path $u_0,u_1,\dots,u_s$ with $u_0=\epsilon$, ending at the
first node $u_s$ all of whose known children have weight $\le W/2$. Such node $u_s$ always exists since any layer $k$ node has no known children. We call
$u_0,u_1,\dots,u_s$ the \emph{spine}. For $i=0,\dots,s$, the children of $u_i$ of weight $\le W/2$ are called the \emph{light children} of $u_i$. These are all of the children of $u_i$ except $u_{i+1}$. Denote by $L_i$ the set of light children of $u_i$, and let
\[
    A = \bigcup_{i=0}^{s} L_i.
\]
Next, we compute the branching polynomials of all nodes in the spine to a sufficient precision. From that, we infer the branching polynomials of the nodes in $A$ to a sufficient precision. We can use these polynomials to recursively compute for each node $x \in A$ the branching polynomials of the layer $k$ nodes in the subtree of $x$.

We get two main benefits from the fact that $w(x) \le \frac{W}{2}$ for all $x \in A$. First, it bounds the depth of the recursion by $\log(W)$. And second, it allows us once again to reduce the precision of the polynomials we work with. 

\begin{figure}[ht]
\centering
\begin{tikzpicture}[
    every node/.style={font=\small},
    spine/.style={circle,draw,thick,fill=black!20,inner sep=1pt,minimum size=6.5mm},
    light/.style={circle,draw,fill=white,inner sep=1pt,minimum size=5.5mm},
    heavyedge/.style={very thick},
    lightedge/.style={},
    tri/.style={draw,fill=black!5},
    rec/.style={font=\scriptsize\itshape},
    wlbl/.style={font=\scriptsize},
    lbl/.style={font=\footnotesize},
    scale=1
]
 
\node[spine] (u0) at (0,0)    {$\epsilon$};
\node[spine] (u1) at (0,-2.2) {$u_1$};
\node[spine] (u2) at (0,-4.4) {$u_2$};
\node[spine] (us) at (0,-6.6) {$u_s$};
 
\draw[heavyedge] (u0) -- (u1);
\draw[heavyedge] (u1) -- (u2);

\draw[heavyedge] (u2) -- ($(u2)!0.3!(us)$);
\draw[very thick,dotted] ($(u2)!0.3!(us)$) -- ($(u2)!0.7!(us)$);
\draw[heavyedge] ($(u2)!0.7!(us)$) -- (us);

\node[wlbl,anchor=east] at (-0.28,-1.95) {$>W/2$};
\node[wlbl,anchor=east] at (-0.28,-4.15) {$>W/2$};
\node[wlbl,anchor=east] at (-0.28,-6.35) {$>W/2$};
 
\node[light] (c0) at (-2.6,-1.0) {};
\draw[lightedge] (u0) -- (c0);
\draw[tri] (c0.south) -- ++(-0.95,-1.0) -- ++(1.9,0) -- cycle;
\node[rec] at ($(c0)+(0,-0.95)$) {recurse};
\node[wlbl,anchor=east] at ($(c0)+(-0.5,0)$) {$\le W/2$};
 
\node[light] (c1a) at (-2.6,-3.2) {};
\draw[lightedge] (u1) -- (c1a);
\draw[tri] (c1a.south) -- ++(-0.95,-1.0) -- ++(1.9,0) -- cycle;
\node[rec] at ($(c1a)+(0,-0.95)$) {recurse};
\node[wlbl,anchor=east] at ($(c1a)+(-0.5,0)$) {$\le W/2$};
 
\node[light] (csa) at (-2.6,-7.6) {};
\draw[lightedge] (us) -- (csa);
\draw[tri] (csa.south) -- ++(-0.95,-1.0) -- ++(1.9,0) -- cycle;
\node[rec] at ($(csa)+(0,-0.95)$) {recurse};
\node[wlbl,anchor=east] at ($(csa)+(-0.5,0)$) {$\le W/2$};
 
\node[light] (c1b) at (2.6,-3.2) {};
\draw[lightedge] (u1) -- (c1b);
\draw[tri] (c1b.south) -- ++(-0.95,-1.0) -- ++(1.9,0) -- cycle;
\node[rec] at ($(c1b)+(0,-0.95)$) {recurse};
\node[wlbl,anchor=west] at ($(c1b)+(0.5,0)$) {$\le W/2$};
 
\node[light] (c2) at (2.6,-5.4) {};
\draw[lightedge] (u2) -- (c2);
\draw[tri] (c2.south) -- ++(-0.95,-1.0) -- ++(1.9,0) -- cycle;
\node[rec] at ($(c2)+(0,-0.95)$) {recurse};
\node[wlbl,anchor=west] at ($(c2)+(0.5,0)$) {$\le W/2$};
 
\node[light] (csb) at (2.6,-7.6) {};
\draw[lightedge] (us) -- (csb);
\draw[tri] (csb.south) -- ++(-0.95,-1.0) -- ++(1.9,0) -- cycle;
\node[rec] at ($(csb)+(0,-0.95)$) {recurse};
\node[wlbl,anchor=west] at ($(csb)+(0.5,0)$) {$\le W/2$};

\node[wlbl] at (0,-7.1) {\itshape all children light};
 
\node[spine,minimum size=5mm] (kS) at (2.55,-0.1) {};
\node[anchor=west,lbl] at (2.95,-0.1) {spine node};
\node[light,minimum size=5mm] (kL) at (2.55,-0.85) {};
\node[anchor=west,lbl] at (2.95,-0.85) {light child};
\draw[rounded corners] (2.2,0.3) rectangle (4.8,-1.25);
 
\end{tikzpicture}
\caption{The spine and the light children}
\label{fig:spine}
\end{figure}
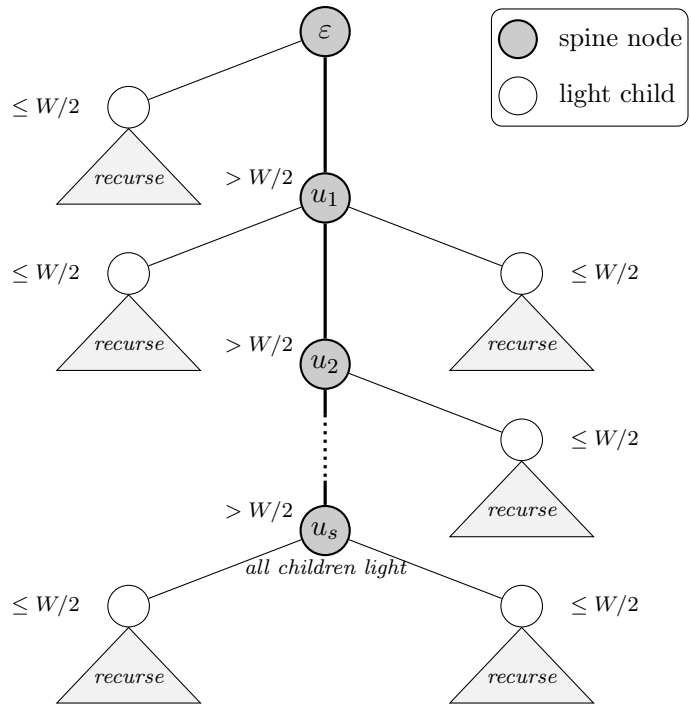

\begin{lemma}\label{lem:A_weight}
    $\sum_{x \in A}w(x) \le W$
\end{lemma}
\begin{proof}
For \(0\le i\le s\), let
\[
    S_i=\sum_{j=i}^s\sum_{x\in L_j} w(x).
\]
We claim that
\[
    S_i \le w(u_i)
\]
for every \(i=0,\ldots,s\), and prove it by downward induction on $i$. For \(i=s\), this follows from
\cref{lem:weights_of_children_sum_bound}, since \(L_s\) is the set of all
children of \(u_s\).

Now suppose \(i<s\). The children of \(u_i\) consist of the spine child
\(u_{i+1}\), together with the light children \(L_i\). Therefore
\[
    \sum_{x\in L_i} w(x) + w(u_{i+1})
    \le w(u_i)
\]
by \cref{lem:weights_of_children_sum_bound}. By the induction hypothesis,
\(S_{i+1}\le w(u_{i+1})\), so
\[
    S_i
    =
    \sum_{x\in L_i}w(x)+S_{i+1}
    \le
    \sum_{x\in L_i}w(x)+w(u_{i+1})
    \le
    w(u_i).
\]
Taking \(i=0\), we get
\[
    \sum_{x\in A}w(x)=S_0\le w(\epsilon)\le W.
\]
\end{proof}

Next, we define $q_i = \sum_{x\in L_i}w(x)$ for $i = 0,\dots,s$. From \cref{lem:A_weight},
\begin{equation}\label{q_sum_small}
    q_0 + \ldots + q_s \le W.
\end{equation}

\begin{algorithm}[H]
    \caption{Computing spine branching polynomials}\label{alg:computing_spine_polys}
    
    Input: $b,W \in \mathbb{N}$, a range $[l,r]$ with $0\le l \le r \le s \le b$, the integers $q_l,\ldots,q_r$, and the base-$p$ digits written on the path connecting the spine nodes $u_l$ to $u_r$. We are also given $f_{u_l}^*$ to precision $\operatorname{per}(l,r) := \max_{l\le i \le r}(W(i-l) +q_i(b-i))$.
    
    Output: The polynomials $f_{u_i}^*$ to precision $q_i(b-i)$ for $i \in [l,r]$.

    \noindent\rule{\textwidth}{0.4pt}
    \begin{algorithmic}[1]
    \If{$l=r$}
        \State Output $f_{u_l}^*$ and return
    \EndIf
    \State Set $mid = \lceil \frac{(l+r)}{2}\rceil$. 
    \State Compute $f_{u_{mid}}^*$ up to precision $\operatorname{per}(mid,r)$ using \cref{lem:computing_branching_polys} \label{step:compute_middle_spine_poly}
    \State Recurse on $[mid,r]$ with $f_{u_{mid}}^*$ \label{step:recurse_right}
    \State Recurse on $[l,mid-1]$ with $f_{u_l}^* \bmod  p^{\operatorname{per}(l,mid-1)}$ \label{step:recurse_left}
    \end{algorithmic}
\end{algorithm}

\paragraph{Correctness.} Assume that $l=r$. We have $\operatorname{per}(l,l) = q_l(b-l)$, so the precision of $f_{u_{l}}^*$ in the input matches the one in the output. Thus we return a correct result in this case. From now on assume $l<r$.

We next show that the conditions of \cref{lem:computing_branching_polys} are met when it is invoked in line \ref{step:compute_middle_spine_poly}. Indeed, 
$
\operatorname{per}(mid,r) = \max_{mid \le i \le r}(W(i-mid)+q_i(b-i)) \le \max_{l \le i \le r}(W(i-mid)+q_i(b-i)) = \operatorname {per}(l,r) - W(mid-l) \le \operatorname{per}(l,r) - w(u_{l+1})(mid-l)
$.

Thus we manage to compute $f_{u_{mid}}^*$ to a sufficient precision for the recursive call in line \ref{step:recurse_right}, and the desired polynomials are output in the recursive calls in lines \ref{step:recurse_right} and \ref{step:recurse_left}.

\paragraph{Running time.} 
We prove that the algorithm runs in $\tilde{O}(nbW)$.

By \cref{lem:computing_branching_polys}, line \ref{step:compute_middle_spine_poly} takes $\tilde{O}(n\operatorname{per}(l,r)) \le \tilde{O}(n((r-l)W+\max_{l\le i\le r}(q_i)b)) \le \tilde{O}(n((r-l)W+b\sum_{l\le i\le r}q_i))$. Let $T(l,r)$ be the running time of this algorithm. For $l<r$, we have
\[
T(l,r) \le T(l,mid-1) + T(mid,r) + \tilde{O}(n((r-l)W+b\sum_{l\le i\le r}q_i)) .
\]
Observe that at each level of the recursion the sum of lengths of the intervals equals exactly $r-l+1$. Thus, the contribution of the term $\tilde{O}(n(r-l)W)$ at each level is the same. Additionally, since we sum over all $q_i$ and the intervals are disjoint, the term  $ \tilde{O}(n(b\sum_{l\le i\le r}q_i))$ is again the same at all the levels of the recursion. Consequently,
\[
T(l,r) = \tilde{O}(n((r-l)W+b\sum_{l\le i\le r}q_i)) \le \tilde{O}(nbW),
\]
using $\sum_{l\le i\le r}q_i \le W$. The factor $\log s$ which comes from the number of levels of the recursion is absorbed into the $\tilde{O}$.

\begin{algorithm}[H]
    \caption{Computing branching polynomials of layer $k$}\label{alg:computing_branching_polys}
    
    Input: $b,W \in \mathbb{N}$, $k \in \Z_{\ge 0}$, $f \in \Z[x]$ given to precision $Wb$ with $v(f) = 0$ and $W \ge \deg f_p^*$. We are also given layers $0$ to $k$ of the lifting tree of $f$, given as a list of edges with the base $p$ digits written on them, together with all node weights.

    Output: The branching polynomials $f_{x}^*$ of all layer $k$ nodes given to precision $w(x)(b-k)$.
    \noindent\rule{\textwidth}{0.4pt}
    \begin{algorithmic}[1]
    \State Find the spine $u_0,\ldots,u_s$ and compute the values $q_0,\dots,q_s$. \label{step:find_spine}
    \State Use \cref{alg:computing_spine_polys} with parameters $(b,W,l=0, r=s,q_0,\dots,q_s)$, the digits on the path from $u_0$ to $u_s$, and $f_{u_l}^* = f_{u_0}^*=f_{\epsilon}^*=f$ to compute the polynomials $f_{u_i}^*$ to precision $q_i(b-i)$ for $i=0,\dots,s$.\label{step:alg2}
    \For{$i=0,\dots,s$}
        \For{$x : x \text{ is a light child of $u_i$}$}
            \State Compute $f_x^*$ to precision $w(x)(b-i-1)$ using \cref{lem:computing_branching_polys} with $f_{u_i}^* \bmod p^{w(x)(b-i)}$\label{step:computing_child_branching_poly}
            \State Recursively use this algorithm with $b-i-1 \rightarrow b$, $w(x) \rightarrow W$, $k-i-1 \rightarrow k$, $f_x^* \rightarrow f$ to output the branching polynomials of all layer $k$ nodes contained in the subtree of $x$. \label{step:recursive_call_alg3}
        \EndFor
    \EndFor
    \State If $s=k$, compute $f_{u_k}^*$ up to precision $w(u_k)(b-k)$ using \cref{lem:computing_branching_polys} with $f$ and output it (if $k=0$, then $u_k=\epsilon$ and $f_{u_k}^*=f$). \label{step:comp_u_s_branching_poly}
    \end{algorithmic}
\end{algorithm}

\paragraph{Correctness.} We begin by verifying that $f$ is known to a sufficient precision when given to \cref{alg:computing_spine_polys} in line \ref{step:alg2}. Indeed, $\operatorname{per}(0,s) = \max_{0 \le i \le s}(Wi+q_i(b-i)) \le \max_{0 \le i \le s}(Wi+W(b-i)) = Wb$.

Next, note that we know $f_{u_i}^*$ to precision $q_i(b-i) \ge w(x)(b-i)$ by \cref{alg:computing_spine_polys}, so it is known to precision $w(x)(b-i)$ when given to \cref{lem:computing_branching_polys} in line \ref{step:computing_child_branching_poly}. Thus $f_x^*$ is indeed computed to precision $w(x)(b-i)-w(x) = w(x)(b-i-1)$ in line \ref{step:computing_child_branching_poly}. 

We now prove that the parameters given in the recursive call to \cref{alg:computing_branching_polys} in line \ref{step:recursive_call_alg3} satisfy the input conditions. Indeed, $f_x^*$ is known to precision $w(x)(b-i-1)$, and $v(f_x^*)=0$ is trivial. $w(x) \ge \deg (f_x)_p^*$ is true by \cref{lem:deg_le_weight}. We note that since $x$ is in layer $i+1$, the layer $k-i-1$ nodes of the lifting tree of $f_x^*$ are exactly the layer $k$ nodes of the lifting tree of $f$ in the subtree of $x$. And since $(b-i-1) - (k-i-1) = b-k$, the output of the recursive call is of sufficient precision. 

Note that the use of \cref{lem:computing_branching_polys} in line \ref{step:comp_u_s_branching_poly} is valid when $k\ge1$. Indeed, $f$ is known to precision $Wb$ and $Wb - w(u_1)k \ge Wb-Wk=W(b-k)\ge w(u_k)(b-k)$.

Also note that each layer $k$ node other than possibly $u_s$ is contained in exactly one subtree of a node in $A$. And $u_s$ is not contained in any of them, and is output if and only if $s=k$. Thus each branching polynomial of a layer $k$ node appears exactly once in the output.

\paragraph{Running time.} 
We prove that the algorithm runs in $\tilde{O}(nbW)$.

Line \ref{step:find_spine} requires a simple scan of the tree and direct computations, which can be done in $\tilde{O}(bW)$ time since the spine length is bounded by $k \le b$, and the number of children of each node is bounded by $W$ (since the number of nodes in any layer is bounded by $W$).

According to the running time analysis of \cref{alg:computing_spine_polys}, line \ref{step:alg2} runs in $\tilde{O}(nbW)$. 

Line \ref{step:computing_child_branching_poly} runs in $\tilde{O}(nw(x)(b-i)) \le \tilde{O}(nw(x)b)$, which sums to $\tilde{O}(nbW)$ as $\sum_{x \in A}w(x) \le W$. Note that reducing $f_{u_i}^*$ modulo $p^{w(x)(b-i)}$ takes negligible time, as the coefficients are stored as base $p$ integers and reducing an integer modulo a power of $p$ just means considering only a prefix of its digits. 

Note that $n$ stays fixed in our recursive calls, and $b$ does not increase. Let $T(W)$ be the maximum running time of \cref{alg:computing_branching_polys}, fixing $n$ and an upper bound on $p$.  Then line \ref{step:recursive_call_alg3} takes at most $T(w(x))$ time.

Since $f$ is of degree $n$ and is given to precision $bW$, by \cref{lem:computing_branching_polys}, Line \ref{step:comp_u_s_branching_poly} takes $\tilde{O}(nbW)$ time.

Overall, we have 
\[
T(W) \le \tilde{O}(nbW) + \sum_{x \in A} T(w(x))
\]
Recall that $\sum_{x \in A}w(x) \le W$ and $w(x) \le \frac{W}{2}$ for all $x \in A$.  Therefore there are $O(\log W)$ levels of recursion and each level contributes $\tilde{O}(nbW) $ and we thus get (with a different $\tilde{O}$)
\[
T(W) = \tilde{O}(nbW).
\]

We now give the algorithm that computes layers $0$ to $b$ of the lifting tree. The structure of the algorithm is
illustrated in the following figure.

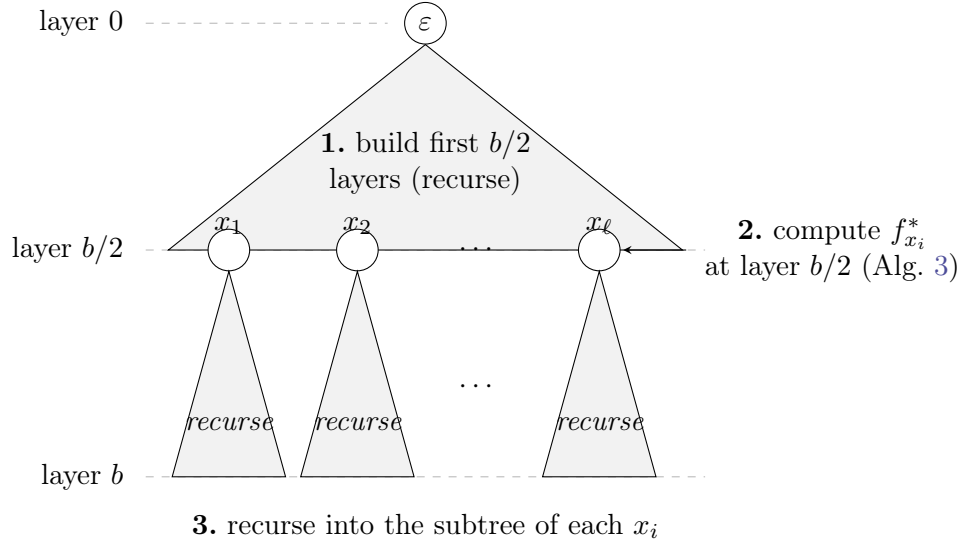
\begin{figure}[H]
\centering
\begin{tikzpicture}[
    every node/.style={font=\small},
    nd/.style={circle,draw,fill=white,inner sep=1pt,minimum size=5.5mm},
    tri/.style={draw,fill=black!5},
    guide/.style={dashed,gray!55,thin},
    rec/.style={font=\small\itshape},
    stp/.style={font=\small,align=center},
    llbl/.style={font=\small},
    scale=1
]
 
\draw[guide] (-3.7,0) -- (-0.5,0);
\draw[guide] (-3.7,-3) -- (3.7,-3);
\draw[guide] (-3.7,-6) -- (3.7,-6);
\node[llbl,anchor=east] at (-3.85,0)  {layer $0$};
\node[llbl,anchor=east] at (-3.85,-3) {layer $b/2$};
\node[llbl,anchor=east] at (-3.85,-6) {layer $b$};
 
\node[nd] (eps) at (0,0) {$\epsilon$};
\draw[tri] (eps.south) -- (-3.4,-3) -- (3.4,-3) -- cycle;
\node[stp] at (0,-1.85) {\textbf{1.}~build first $b/2$\\ layers (recurse)};
 
\node[nd] (x1) at (-2.6,-3) {};
\node[nd] (x2) at (-0.9,-3) {};
\node[nd] (xl) at ( 2.3,-3) {};
\node at (0.7,-3) {$\cdots$};
\node[rec] at (-2.6,-2.68) {$x_1$};
\node[rec] at (-0.9,-2.68) {$x_2$};
\node[rec] at ( 2.3,-2.68) {$x_\ell$};
 
\draw[tri] (x1.south) -- (-3.35,-6) -- (-1.85,-6) -- cycle;
\node[rec] at (-2.6,-5.3) {recurse};
\draw[tri] (x2.south) -- (-1.65,-6) -- (-0.15,-6) -- cycle;
\node[rec] at (-0.9,-5.3) {recurse};
\draw[tri] (xl.south) -- (1.55,-6) -- (3.05,-6) -- cycle;
\node[rec] at (2.3,-5.3) {recurse};
\node at (0.7,-4.8) {$\cdots$};
 
\draw[->,>=stealth,shorten >=1pt] (3.45,-3) -- (xl.east);
\node[stp,anchor=west] at (3.55,-3)
    {\textbf{2.}~compute $f^*_{x_i}$\\ at layer $b/2$ (Alg.~\ref{alg:computing_branching_polys})};
 
\node[stp] at (0,-6.65) {\textbf{3.}~recurse into the subtree of each $x_i$};
 
\end{tikzpicture}
\caption{High-level illustration of the algorithm}
\label{fig:tree-recursion}
\end{figure}

\begin{algorithm}[H]
    \caption{Tree building algorithm}\label{alg:Tree building algorithm}
    
    Input: A prime $p$, integers $n,b,W \in \mathbb{N}$, a polynomial $f \in \mathbb{Z}[x]$ of degree $n$ with $v(f)=0$, $\deg(f_p) \le W$, given to precision $bW$.
    
    Output: layers $0$ to $b$ of the lifting tree as a rooted tree with edge
labels in \(\{0,\ldots,p-1\}\), together with the weights of all nodes.

    \noindent\rule{\textwidth}{0.4pt}
    \begin{algorithmic}[1]
    \If{$b=1$}
        \State Return the roots of $f_p$ with their multiplicities using \cref{lem:find_roots_and_multiplicities} \label{step:b=1}
    \EndIf
    \State Recursively run this algorithm with $(p,n,\lceil\frac{b}{2}\rceil,W,f \bmod p^{W\lceil\frac{b}{2}\rceil})$ to construct layers $0$ to $\lceil\frac{b}{2}\rceil$ of the lifting tree and their weights. Let the nodes of layer $\lceil\frac{b}{2}\rceil$ be $x_1,\ldots,x_l$. \label{step:recurse_b/2}
    \State Compute the branching polynomials $f_{x_i}^*$ to precision $w(x_i)\lfloor\frac{b} {2}\rfloor$ for $i=1,\dots,l$
    by applying \cref{alg:computing_branching_polys} to the already constructed layers $0$ to $\lceil\frac{b}{2}\rceil$ of the lifting tree, with parameters $(b,W,k=\lceil\frac{b}{2}\rceil,f)$. \label{step:use_alg3}
    \State For each $i$, recursively run this algorithm with $(p,n,\lfloor\frac{b} {2}\rfloor,w(x_i),f_{x_i}^* \bmod p^{\lfloor\frac{b} {2}\rfloor{w(x_i)}})$ to construct layers $0$ to $\lfloor\frac{b} {2}\rfloor$ \label{step:recursion_alg4}
 of the lifting tree of \(f_{x_i}^*\), along with the weights of its nodes, and attach the resulting
edge-labeled tree below \(x_i\). \label{step:recurse_on_layer_b/2}
    \end{algorithmic}
\end{algorithm}

\paragraph{Correctness.} The only non-trivial detail to verify is that the precision of $f_{x_i}^*$ in the recursive call in line \ref{step:recursion_alg4} is sufficient. $f_{x_i}^*$ is known to precision $\lfloor \frac{b}{2} \rfloor w(x_i) \ge \lfloor \frac{b}{2} \rfloor \deg(f_{x_i})_p^*$ by \cref{lem:deg_le_weight}. Thus the conditions in the recursive call to \cref{alg:Tree building algorithm} are met.  

\paragraph{Running time.}
We prove that the algorithm runs in $\tilde{O}(nbW+W\sqrt{p})$.

We now analyze the total cost of all executions of line
\ref{step:b=1}. Each such execution corresponds to some node \(v\) in one of
the layers \(0,\ldots,b-1\): at that node, the algorithm finds the roots of
\((f_v)_p^*\), together with their multiplicities. Let \(d_v\) be the number of
children of \(v\), equivalently the number of distinct roots of \((f_v)_p^*\)
in \(\mathbb F_p\).

By \cref{lem:find_roots_and_multiplicities}, the cost at \(v\) is
\[
    \tilde{O}(n\log^2 p+\sqrt{d_vp}),
\]
and if \(d_v\le 1\), the \(\sqrt{d_vp}\) term may be omitted.

First, we bound the contribution of the \(n\log^2 p\) terms. The algorithm
invokes line \ref{step:b=1} at most once for each node in layers
\(0,\ldots,b-1\). Each layer contains at most \(W\) nodes, so the number of such
calls is at most \(bW\). Hence these terms contribute
\[
    \tilde{O}(nbW).
\]

It remains to bound the square-root terms. Let \(L\) be the set of leaves of the
finite tree computed up to layer \(b\), where a node is also considered a leaf if
it has no children before layer \(b\). The weights of the leaves have total
weight at most \(W\): this follows by repeatedly using the fact that the sum of
the weights of the children of any node is at most the weight of the node.
Since every leaf has weight at least \(1\), we get
\[
    |L|\le W.
\]

In any finite rooted tree,
\[
    \sum_{v:\,d_v\ge 2}(d_v-1)\le |L|-1\le W.
\]
For \(d_v\ge 2\), we have \(\sqrt{d_v}\le 2(d_v-1)\). Hence
\[
    \sum_{v:\,d_v\ge 2}\sqrt{d_v}
    \le
    2\sum_{v:\,d_v\ge 2}(d_v-1)
    \le 2W.
\]
Thus the total contribution of the square-root terms is
\[
    \tilde{O}(W\sqrt p).
\]

Altogether, all executions of line \ref{step:b=1} take
\[
    \tilde{O}(nbW+W\sqrt p)
\]
time.

Now we fix $n$, and let $T(W,b)$ be the maximal running time of lines \ref{step:recurse_b/2}-\ref{step:recurse_on_layer_b/2}, ignoring the running time of line \ref{step:b=1}.

Line \ref{step:recurse_b/2} takes $T(W,\frac{b}{2})$ time.

By the running time analysis of \cref{alg:computing_branching_polys}, line \ref{step:use_alg3} takes $\tilde{O}(nbW)$ time.

Line \ref{step:recurse_on_layer_b/2} takes $ \sum_{i=1}^lT(w(x_i),\frac{b}{2})$ time.

Overall, we get
\[
T(W,b) \le \tilde{O}(nbW)+T(W,\frac{b}{2})+\sum T(w(x_i),\frac{b}{2})
\]
and recall that $\sum w(x_i) \le W$. There are $\log b$ levels for the recursion. At each level the sum of weights times the relevant $b$ is always upper bounded by $bW$. Thus,
\[
T(W,b) = \tilde{O}(nbW)
\]
as desired.

We now give the full algorithm that proves \cref{thm:lifting_roots}. First, we find all roots of $f_p$, which is layer $1$ of the lifting tree of $f$. Then, we handle simple roots and repeated roots separately. We perform parallel Hensel lifting using multipoint evaluation to lift all simple roots; and use \cref{alg:Tree building algorithm} to lift all repeated roots.

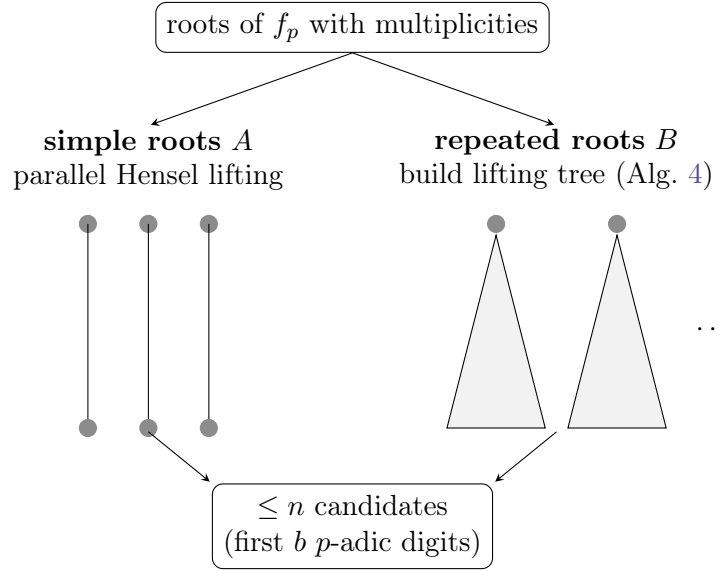
\begin{figure}[H]
\centering
\begin{tikzpicture}[
    every node/.style={font=\small},
    box/.style={draw,rounded corners,align=center,inner sep=4pt,font=\small},
    hdr/.style={align=center,font=\small},
    dot/.style={circle,fill=black!45,inner sep=0pt,minimum size=2.4mm},
    tri/.style={draw,fill=black!5},
    liftline/.style={thin},
    ar/.style={-stealth,shorten >=1pt},
    scale=1
]

\node[box] (src) at (0,0.4) {roots of $f_p$ with multiplicities};

\draw[ar] (src.south) -- (-2.7,-0.9);
\draw[ar] (src.south) -- ( 2.7,-0.9);

\node[hdr] at (-2.7,-1.35) {\textbf{simple roots} $A$\\[-1pt] {\small parallel Hensel lifting}};
\foreach \x in {-3.5,-2.7,-1.9} {
    \node[dot] at (\x,-2.2) {};
    \draw[liftline] (\x,-2.2) -- (\x,-4.9);
    \node[dot] at (\x,-4.9) {};
}

\node[hdr] at (2.7,-1.35) {\textbf{repeated roots} $B$\\[-1pt] {\small build lifting tree (Alg.~\ref{alg:Tree building algorithm})}};
\node[dot] at (1.9,-2.2) {};
\draw[tri] (1.9,-2.35) -- (1.25,-4.9) -- (2.55,-4.9) -- cycle;
\node[dot] at (3.5,-2.2) {};
\draw[tri] (3.5,-2.35) -- (2.85,-4.9) -- (4.15,-4.9) -- cycle;
\node at (4.8,-3.6) {$\cdots$};

\node[box] (out) at (0,-6.2) {$\le n$ candidates\\ (first $b$ $p$-adic digits)};
\draw[ar] (-2.7,-4.95) -- (out.north west);
\draw[ar] ( 2.7,-4.95) -- (out.north east);

\end{tikzpicture}
\caption{The finishing step. The roots of $f_p$ are found with their multiplicities
(\cref{lem:find_roots_and_multiplicities}) and split by multiplicity. A \emph{simple} root
(multiplicity $1$) has a unique lift, so all of $A$ is lifted in parallel by Hensel lifting,
doubling the $p$-adic precision over $\lceil\log b\rceil$ rounds of multipoint evaluation. A
\emph{repeated} root (multiplicity $>1$) branches, so each $x\in B$ is lifted by building the
first $b$ layers of its lifting tree with \cref{alg:Tree building algorithm}. The union of the
two families is a set of at most $n$ candidates containing the first $b$ $p$-adic digits of
every root of $f$.}
\label{fig:finishing}
\end{figure}

\begin{algorithm}[H]
    \caption{Finishing the lifting step}\label{alg:finish_lifting_roots}
    
    Input: A prime $p$, integers $n,b,t \in \mathbb{N}$, a polynomial $f$ of degree $n$ such that the sum of multiplicities of the repeated roots of $f_p$ is bounded by $t$, and $v(f)=0$. $f$ is given to precision $tb$. 
    
    Output: Layer $b$ of the lifting tree.

    \noindent\rule{\textwidth}{0.4pt}
    \begin{algorithmic}[1]
        \State Invoke \Cref{lem:find_roots_and_multiplicities}
        to find all roots of $f_p$ and their multiplicities. If $b=1$, output the roots of $f_p$ and terminate. Let $A$ be the set of roots of multiplicity 1, and $B$ the set of roots with multiplicity $>1$. \label{step:find_f_p_roots}
        \For{$x \in B$}
            \State Compute $f_x^*$ to precision $w(x)(b-1)$ using \cref{lem:computing_branching_polys} with $f \bmod p^{w(x)b}$ \label{step:Compute_B_branching_polys}
            \State Run \cref{alg:Tree building algorithm} with $(p,n,b-1,w(x),f_x^*)$ to compute the nodes of layer $b-1$ of the lifting tree of $f_x^*$. Prefix each of them by the digit $x$, and add them to a set $C$ (initially empty). \label{step:final_run_alg4}
        \EndFor
        \For{$i = 1,\dots,\lceil\log(b)\rceil$}
            \State Let $k_i=\min(b,2^i)$. Use \tool{comp:multipoint_evaluation} to compute the values $f(x)$ and $f'(x)$ in $\mathbb Z/p^{k_i}\mathbb Z$ for all $x \in A$. \label{step:hensel_eval}
            \State Replace every $x\in A$ by $x-f(x)\cdot f'(x)^{-1} \bmod p^{k_i}$. \label{step:hensel_update}
        \EndFor
        \State Output $A\cup C$.
    \end{algorithmic}
\end{algorithm}

This algorithm proves \cref{thm:lifting_roots}. We require $v(f)=0$ which follows from the fact that $p$ does not divide the leading coefficient of $f$. And we only need $f$ to precision $tb$.

\paragraph{Correctness.} 
We have $w(x) \le t$ for $x \in B$. Thus, $f$ is known to precision $w(x)b$, which is the precision we need for line \ref{step:Compute_B_branching_polys} since $w(x)b - w(x) = w(x)(b-1)$.

We verify that the conditions on the input for \cref{alg:Tree building algorithm} are met. By \cref{lem:deg_le_weight}, we have $\deg (f_x)_p^* \le w(x)$. And $f_x^*$ is indeed known to precision $w(x)(b-1)$.

Thus, for every $x\in B$, layer $b-1$ of the lifting tree of $f_x^*$, prefixed by the digit $x$, is added to $C$. These are exactly the layer $b$ nodes of the lifting tree of $f$ which descend from $x$.

The rest of the algorithm performs Hensel lifting (Newton iteration) on the roots in $A$. Every $x\in A$ is a simple root of $f_p$, so $p\nmid f'(x)$ and $f'(x)$ is invertible modulo every power of $p$. By the quadratic convergence of Newton iteration \cite[Lemma~9.21]{vzgg2013}, if $f(x)\equiv 0\pmod {p^{k_{i-1}}}$ before iteration $i$ (where $k_0=1$), then after iteration $i$ the updated $x$ satisfies $f(x)\equiv 0 \pmod {p^{k_i}}$ and is congruent to the old $x$ modulo $p^{k_{i-1}}$. Since $k_{\lceil \log b\rceil}=b$, each simple root of $f_p$ is lifted to its unique lift modulo $p^b$, which is the unique layer $b$ node that descends from it.

\paragraph{Running time.} 
We prove that the algorithm runs in $\tilde{O}(nbt+(t+\sqrt{n})\sqrt{p})$ time.

Line \ref{step:find_f_p_roots} takes $\tilde{O}(n+\sqrt{np})$ time by \cref{lem:find_roots_and_multiplicities}. 

Line \ref{step:Compute_B_branching_polys} costs $\tilde{O}(nbw(x))$, which sums to $\tilde{O}(nbt)$ since $\sum_{x \in B}w(x) \le t$.

Line \ref{step:final_run_alg4} takes $\tilde{O}(w(x)(nb+\sqrt{p}))$ by \cref{lem:build_tree_cost}, which sums to $\tilde{O}(t(nb+\sqrt{p}))$.

It remains to analyze the loop over the simple roots \(A\). In the \(i\)-th iteration we work over
\(\mathbb Z/p^{k_i}\mathbb Z\), where \(k_i\le 2^i\). Since \(|A|\le n\), multipoint evaluation of both
\(f\) and \(f'\) at the points of \(A\) (line \ref{step:hensel_eval}) costs
\[
    \tilde{O}(n2^i)
\]
bit operations. Each of the at most $n$ inversions and Newton updates in line \ref{step:hensel_update} takes \(\tilde{O}(2^i)\) time, so this line also takes \(\tilde{O}(n2^i)\) time. Therefore the total cost of this loop is
\[
    \sum_{i=1}^{\lceil \log(b)\rceil}\tilde{O}(n2^i)
    =
    \tilde{O}(nb).
\]

Summing up the steps, the algorithm takes $\tilde{O}(nbt + (\sqrt{n}+t)\sqrt{p})$, as desired.

\end{proof}

\section{Verifying roots}\label{sec:verifying}

We present two algorithms for verifying candidate roots. The first runs in $\tilde{O}(nb+n^2)$, which is fast when $b$ is large. The second algorithm runs in $\tilde{O}(nb^2)$, which is fast when $b$ is small. Both algorithms work after minor adjustments to the input. They assume that the constant coefficient of $f$ is nonzero, that the candidates $x_i$ are positive, and that $m=n$. At the end of the section we explain how these algorithms imply  \cref{thm:verifying_roots}.

Both algorithms apply simple divisibility tests to filter the candidate set; the analysis
shows that these tests leave either few candidates or candidates of small size.

Both algorithms are simple to state, but their correctness and running-time analysis are involved. The large-$b$ algorithm is self-contained, whereas the small-$b$ algorithm uses a lemma from the root-finding section.

\subsection{Algorithm for large $b$}

We give an algorithm for verifying roots in $\tilde{O}(nb+n^2)$. 

First note that we may assume $b\ge n$. Indeed, if $b < n$, replacing $b$ by $n$ preserves the conditions $\|f\|_\infty < 2^b$ and $x_i < 2^b$ and it does not change the target bound $\tilde{O}(nb+n^2)$.

The algorithm is motivated by the following observation: 
for any integers $x,y$, we have $x-y \mid f(x)-f(y)$. In particular, if $f(x_i)=0$, then $x_i-y \mid f(y)$.

The idea is to repeatedly test the condition $x_i - y \mid f(y)$ for small values of $y$, eliminating candidate roots as we proceed.

This test gives no information when $f(y)=0$, as $x_i - y \mid f(y)$ holds trivially. We therefore begin with a preparation step that finds a segment of small non-negative integers where $f$ doesn't vanish.

\begin{lemma}\label{lem:find_segment_f_doesn't_vanish}
    We can find an integer $0 \le a \le 2n^2$ such that $f(x) \neq 0$ for every
    integer $x \in [a, a+2n-1]$, in $\tilde{O}(n^2 + nb)$ time.
\end{lemma}
\begin{proof}
        We first give the idea of the proof, then explain how to implement it efficiently. The idea is to find the roots of $f$ modulo a prime $p \sim 2n^2$, and then choose an interval that contains none of those roots. Specifically:

    Set $u = 2n^2 + 2n$. We find the smallest prime $p \ge u$ using the sieve of \tool{comp:sieve}. By the prime number theorem, we have $p = O(u) =O(n^2)$. 

    Recall that $f_p^*$ is the image of $\frac{f}{p^{v_p(f)}}$ in $\F_p$. Calculate $f_p^*$ in $\tilde{O}(nb)$ time and compute the
    root set $R \subseteq \mathbb{F}_p$ of $f_p^*$ using \cref{comp:roots_mod_p}. As
    $p = \tilde{O}(n^2)$, this takes $\tilde{O}(\sqrt{pn}) = \tilde{O}(n^{3/2})$ time.

    Let $U = \{0, 1, \dots, u-1\}$. Since $p \ge u$, $U$'s elements are distinct modulo $p$. Deleting the elements of $U$ which are roots of $f_p^*$ leaves at least
    $u - n = 2n^2 + n$ integers, forming at most $n+1$ maximal segments of consecutive integers.
    By the pigeonhole principle the longest segment has at least
    $\lceil (2n^2 + n)/(n+1) \rceil = 2n$ integers; let $a$ be its smallest element (found by
    sorting $R$ and scanning, in $\tilde{O}(n)$ time). The segment lies in $U$ and has length
    $\ge 2n$, so $a + 2n - 1 \le u - 1$, giving $0 \le a \le 2n^2$. Finally, every integer
    $x \in [a, a+2n-1]$ lies in $U$ and is not a root of $f_p^*$, hence it is also not a root of $f$.
\end{proof}

We now give our algorithm for verifying roots.

\begin{algorithm}[H]
    \caption{Verifying roots for large $b$}
    \label{alg:algorithm_verify_large_b}
    \raggedright
    
    Input: positive integers $n\le b$, the coefficients $a_0,\ldots,a_n $ of a polynomial $f$ with $a_0 \neq 0$ and $\|f\|_\infty < 2^b$, and $n$ distinct candidate roots $x_1,\ldots,x_n$ such that $1 \le x_i < 2^b$ for $i = 1,\ldots,n$.\\
    Output: The indices of the $x_i$ which are roots of $f$.
    \noindent\rule{\textwidth}{0.4pt}
    \begin{algorithmic}[1]
        \State Call any $x_i$ with $x_i \le 32n^22^{\frac{b}{n}}$ small. Compute $f(x_i)$ for each small $x_i$ using \tool{comp:multipoint_evaluation}. Output the indices of the small candidates for which $f(x_i)=0$. Eliminate all small candidates from the list. \label{step:eliminate_small_x_i}
        
        \State Find an integer $0 \le a \le2n^2$ such that $f$ does not vanish on $[a,a+2n-1]$ using \cref{lem:find_segment_f_doesn't_vanish} \label{step:find_segment}
        \State Compute $f(a),\ldots,f(a+2n-1)$ using \tool{comp:multipoint_evaluation} \label{step:evaluate_segment}
        \For{$y = a, \dots, a+2n-1$}\label{second loop}
            \State Compute $f(y) \bmod (x_{i} - y)$ for all remaining candidates $x_i$ using \tool{comp:multipoint_modulo}
            \State Eliminate any candidate $x_i$ such that  $f(y) \not\equiv 0 \pmod{ x_{i} - y}$.
        \EndFor
        \State Output the indices of all remaining candidates.\label{step:output_candidates}
    \end{algorithmic}
\end{algorithm}

We use the following trivial lemma throughout.

\begin{lemma}\label{lem:f-ub}
    $|f(x)| \le (n+1)\|f\|_\infty \max(1,|x|)^n$.
\end{lemma}
\begin{proof}
    It follows from the fact that $|a_ix^i| \le \|f\|_\infty \max(1,|x|)^n$.
\end{proof}

\subsubsection{Correctness}\label{sec:verify_large_b_correctness}
Note that each modulus $x_i -y$ is positive, as every non-small candidate satisfies $x_i > 32n^22^{\frac{b}{n}} \ge 32n^2 \ge 2n+2n^2 > y$.

We prove that if an index $i$ is output at line \ref{step:output_candidates}, then $x_i$ is a genuine root of $f$. The correctness of the other steps is immediate. 

Let $i$ be an index output at line \ref{step:output_candidates}. We next show that if $f(x_i) \neq 0$ then $x_i$ must be ``small'' (according to line \ref{step:eliminate_small_x_i}), and hence was already eliminated.

Since $x_i$ survived all tests, $f(y)\equiv0\pmod{x_i-y}$ for every
$y=a,\ldots,a+2n-1$. Also $f(x_i)\equiv f(y)\pmod{x_i-y}$, hence
$x_i-y\mid f(x_i)$ for all such $y$. We get 
\[\lcm(x_i-a,\ldots,x_i-a-2n+1) \mid f(x_i).\]

We prove a lower bound on $\lcm(x_i-a,\ldots,x_i-a-2n+1)$ to derive $f(x_i)=0$.

\begin{lemma}\label{lem:lcm_bound}
Let $\ell,x$ be positive integers with $\ell \le x$. Then
\[\binom{x}{\ell} \mid \lcm(x,x-1,\ldots,x-\ell+1).\]
\end{lemma}

\begin{proof}
    Let $p$ be a prime number, and let $k = v_p(\operatorname{lcm}(x,x-1,\ldots,x-\ell+1))$.  For $i = 1,2,\dots$, let $s_i,t_i$ be the number of multiples of $p^i$ in the lists $1,\dots,\ell$ and $x-\ell+1,\dots,x$ respectively. Since both intervals are of length $\ell$, we have $t_i \le s_i+1$. Since $k = \max(v_p(x-\ell+1),\dots,v_p(x))$ we have $t_i = 0$ for $i>k$. Thus
    \[
    v_p\binom{x}{\ell} = \sum_{i=1}^{\infty}(t_i - s_i) \le \sum_{i=1}^{k}(t_i - s_i) \le k = v_p(\operatorname{lcm}(x,x-1,\ldots,x-\ell+1)). \qedhere
    \]
\end{proof}

By \cref{lem:lcm_bound}, we get $\binom{x_i-a}{2n} \mid \lcm(x_i-a,\ldots,x_i-a-2n+1) \mid f(x_i)$. 
Assuming  $f(x_i) \neq 0$ we have
\[
\binom{x_i-a}{2n} \le |f(x_i)| \le (n+1)2^bx_i^n.
\]
Using the bound $\binom{u}{v} \ge (\frac{u}{v})^v$ and the inequality $a \le 2n^2$ we obtain
\[\br{\frac{x_i}{2n}-n}^{2n}\le\br{\frac{x_i-a}{2n}}^{2n} \le \binom{x_i-a}{2n} \le (n+1)2^bx_i^n. \]
If $x_i \ge 4n^2$, then $\frac{x_i}{2n}-n \ge \frac{x_i}{4n}$. Simplifying gives:
\begin{align*}
    &(\frac{x_i}{4n})^{2n} \le (n+1)2^bx_i^n \\ \Rightarrow  \quad & x_i \le (4n)^2(n+1)^{\frac{1}{n}}2^{\frac{b}{n}} \le 32n^22^{\frac{b}{n}}.
\end{align*}
If $x_i < 4n^2$, then trivially $x_i \le 32n^22^{\frac{b}{n}}$. Thus, in both cases we must have $x_i \le 32n^22^{\frac{b}{n}}$. Therefore, $x_i$ must have been eliminated in line \ref{step:eliminate_small_x_i} as claimed.
Consequently, $f(x_i)=0$. Thus, only true roots can pass both the filtering in line~\ref{step:eliminate_small_x_i} and  all of the filtering in the for loop of line~\ref{second loop}.

\subsubsection{Running time}

We prove that the algorithm runs in $\tilde{O}(n^2+nb)$.
The main challenge is to bound the average cost of line~\ref{second loop}. We call each iteration of the
for-loop an \emph{elimination round} and treat each $y\in\{a,\ldots,a+2n-1\}$ as a test.

The following lemma will help us to bound the total bitsize of the candidates of each round.

\begin{lemma}\label{lem:survivor_product}
    Let $0\le k\le 2n$ be an integer and let $z_1,\dots,z_t$ be candidates among $x_1,\dots,x_n$,
    with $t \le k/2$, that survive the first $k$ elimination rounds. Then
    $z_1 z_2 \cdots z_t = 2^{\tilde{O}(b)}$.
\end{lemma}

\begin{proof}
    If $z_i$ survived the test by $y$ then $z_i - y \mid f(y)$, so the lcm of the $z_i - y$ divides
    $f(y)$. Since each round has $f(y) \neq 0$ and $y \le 2n^2 + 2n$, \Cref{lem:f-ub} implies
    \[
        \operatorname{lcm}_{i=1}^{t}(z_i - y) \le |f(y)| \le (n+1)2^b\,(2n^2+2n)^n = 2^{\tilde{O}(n+b)},
    \]
    and multiplying over the first $k$ rounds,
    \[
        \prod_{y=a}^{a+k-1} \operatorname{lcm}_i(z_i - y) \;\le\; 2^{\tilde{O}(k(n+b))}.
    \]
    Our goal is to obtain a lower bound on the left-hand side in terms
    of $Z := z_1 z_2\cdots z_t$. We first discard the part of each lcm built from small primes.
    For a positive integer $x$, write $x = R_{<k}(x)\,R_{\ge k}(x)$, splitting $x$ into its
    \emph{$k$-smooth part} $R_{<k}(x)$ (divisible only by primes $<k$) and its \emph{$k$-rough
    part} $R_{\ge k}(x)$ (divisible only by primes $\ge k$); we use freely that $R_{\ge k}$ is
    multiplicative and that $R_{\ge k}(x) \le x$. This definition naturally extends to rational numbers as well. Setting
    \[
        Q \;=\; \prod_{y=a}^{a+k-1} R_{\ge k}\big(\operatorname{lcm}_i(z_i - y)\big),
    \]
    the bound $R_{\ge k}(x) \le x$ applied termwise gives, from the inequality above,
    \begin{equation}
        Q \;\le\; 2^{\tilde{O}(k(n+b))}. \label{eq:ub-Q}
    \end{equation}
    The rest of the proof bounds $Q$ from below in terms of $Z$.
    
    We first replace the lcm by the product of the $z_i - y$, up to a pairwise-gcd correction,
    using the following lemma (whose proof is deferred for ease of reading).

    \begin{lemma}\label{lem:Rge_lcm}
        For any positive integers $m_1,\dots,m_l$, there exists a positive integer $c$ such that 
        \[
            c\cdot R_{\ge k}\!\left(\frac{m_1 m_2 \cdots m_l}{\prod_{1\le i<j\le l}\gcd(m_i,m_j)}\right)
            \;=\;
            R_{\ge k}\big(\operatorname{lcm}(m_1,\dots,m_l)\big).
        \]
    \end{lemma}

    Applying this with $l=t$ and  $m_i = z_i - y$, and expanding by multiplicativity of $R_{\ge k}$,
    \begin{align*}
        R_{\ge k}\big(\operatorname{lcm}_i(z_i - y)\big)
        &\;\ge\; R_{\ge k}\br{
        \frac{\prod_i (z_i - y)}{\prod_{i<j} \!\big(\gcd(z_i - y,\, z_j - y)\big)}}\\ &\;=\;
        \frac{\prod_i R_{\ge k}(z_i - y)}{\prod_{i<j} R_{\ge k}\!\big(\gcd(z_i - y,\, z_j - y)\big)}.
    \end{align*}
    Taking the product over $y = a, \dots, a+k-1$ and regrouping gives 
    \begin{align}
        Q=\prod_{y=a}^{a+k-1}R_{\ge k}\big(\operatorname{lcm}_i(z_i - y)\big) &\;\ge\; \prod_{y=a}^{a+k-1}  
        \frac{\prod_i R_{\ge k}(z_i - y)}{\prod_{i<j} R_{\ge k}\!\big(\gcd(z_i - y,\, z_j - y)\big)} \notag \\ &\;=\; \frac{\prod_i \prod_{y=a}^{a+k-1} R_{\ge k}(z_i - y)}{\prod_{i<j} \prod_{y=a}^{a+k-1} R_{\ge k}\!\big(\gcd(z_i - y,\, z_j - y)\big)}.\label{eq:Q}
    \end{align}
    Denote 
    \[
         P_i = \prod_{y=a}^{a+k-1}(z_i - y),
        \quad N = \prod_i R_{\ge k}(P_i), \quad \text{and}\quad 
        D = \prod_{i<j} R_{\ge k}\!\Big(\prod_{y=a}^{a+k-1} \gcd(z_i - y,\, z_j - y)\Big).
    \]
    We therefore rewrite \eqref{eq:Q} as
    $Q \ge N/D$.
    We bound $N$ and $D$ in turn.

    For $D$ we use the following lemma, whose proof we postpone.

    \begin{lemma}\label{lem:Rge_gcd}
        For any $1\leq i<j\leq t$, we have 
        \[
            R_{\ge k}\!\left(\prod_{y=a}^{a+k-1}\gcd(z_i - y,\, z_j - y)\right)
            \;\Big|\; R_{\ge k}(z_i - z_j).
        \]
    \end{lemma}

    Thus, each factor of $D$ divides the corresponding factor $R_{\ge k}(z_i - z_j)$. It follows that 
    \[D = \prod_{i<j} R_{\ge k}\!\br{\prod_{y=a}^{a+k-1} \gcd(z_i - y,\, z_j - y)} \;\le\; \prod_{i<j}|z_i - z_j| \;\le\; \prod_{i<j} z_i z_j \;=\; Z^{\,t-1} \]
    
    For $N$ we use a lower bound on the rough part of a falling factorial (again, proof postponed).

    \begin{lemma}\label{lem:R_k_lower_bound}
        Let $k \ge 1$ and let $x$ be an integer with $x \ge k$. Writing
        $P = x(x-1)\cdots(x-k+1)$,
        \[
            R_{\ge k}(P) \;\ge\; \frac{P}{k!\,x^{\pi(k)}}.
        \]
    \end{lemma}
    Applying \cref{lem:R_k_lower_bound} to every $x = z_i-a$ we get
    \[
        N = \prod_i R_{\ge k}(P_i) \ge \prod_i \frac{P_i}{k!\,(z_i-a)^{\pi(k)}}
        \stackrel{(\dagger)}{\ge} \frac{\prod_i (z_i/2)^k}{(k!)^t \prod_i z_i^{\pi(k)}}
        = \frac{2^{-kt}\, Z^{\,k}}{(k!)^t\, Z^{\pi(k)}}.
    \]
    We next justify inequality $(\dagger)$. Since the algorithm has already removed every candidate of
    size at most $32n^2 2^{\frac{b}{n}}$, we have $z_i > 32n^2 2^{\frac{b}{n}} > 8n^2$. Combining with $a + k \le 2n^2 + 2n \le 4n^2$, gives $z_i - k - a \ge \frac{z_i}{2}$. Thus, $P_i \ge ({\frac{z_i}{2}})^k$.
    
    Combining the bounds on $N$ and $D$,
    \[
        Q \;\ge\; \frac{N}{D} \;\ge\; \frac{2^{-kt}\, Z^{\,k}}{(k!)^t\, Z^{\pi(k)}\, Z^{\,t-1}}.
    \]
    Comparing with the upper bound \eqref{eq:ub-Q} and absorbing
    $2^{kt}(k!)^t = 2^{O(k^2 \log k)} = 2^{\tilde{O}(k(n+b))}$ (using $k \le 2n$) yields
    \begin{equation}\label{eq:Z-ub}
        Z^{\,k - \pi(k) - (t-1)} \;\le\; 2^{\tilde{O}(k(n+b))}.
    \end{equation}
    If $k = O(1)$, then $t = O(1)$. Since $z_i < 2^b$, we obtain $Z = 2^{O(b)}$. Otherwise, we can assume that
    $k$ is large enough so that $\pi(k) \le k/4$, and since $t - 1 < k/2$ the exponent on the left of \eqref{eq:Z-ub} is
    at least $k - k/4 - k/2 = k/4 > 0$. Dividing through, \eqref{eq:Z-ub} implies
    \[
        \log Z \;\le\; \frac{\tilde{O}(k(n+b))}{k/4} = \tilde{O}(n+b) = \tilde{O}(b),
    \]
    using $b \ge n$. Hence $z_1 \cdots z_t =Z= 2^{\tilde{O}(b)}$.
\end{proof}

It remains to prove the three lemmas used above.

\begin{proof}[Proof of \cref{lem:Rge_lcm}]
    Fix a prime $q$; we show it appears in the right-hand side at least as often as in the left.
    If $q<k$ then $R_{\ge k}$ removes $q$ from both sides, so assume $q\ge k$, where $R_{\ge k}$
    leaves $v_q$ untouched. Pick $r$ with
    $v_q(m_r) = \max_i v_q(m_i) = v_q(\operatorname{lcm}(m_1,\dots,m_l))$. For every $i\neq r$,
    $v_q(\gcd(m_i,m_r)) = v_q(m_i)$, so discarding from the denominator all gcd factors not
    involving $r$ only increases the quotient:
    \[
        v_q\!\left(\frac{\prod_i m_i}{\prod_{i<j}\gcd(m_i,m_j)}\right)
        \le v_q\!\left(\frac{\prod_i m_i}{\prod_{i\neq r}\gcd(m_i,m_r)}\right)
        = v_q(m_r) = v_q(\operatorname{lcm}(m_1,\dots,m_l)). \qedhere
    \]
\end{proof}

\begin{proof}[Proof of \cref{lem:Rge_gcd}]
    Fix a prime $q$; the case $q<k$ is trivial, so assume $q \ge k$. Since $q \ge k$ divides at
    most one of the $k$ consecutive integers $z_i - a, \dots, z_i - a - (k-1)$, either it divides
    none of the gcds in the product, or there is a unique $y_0$ with
    $q \mid \gcd(z_i - y_0, z_j - y_0)$. In the latter case
    $\gcd(z_i - y_0, z_j - y_0) \mid z_i - z_j$, so
    \[
        v_q\!\left(\prod_{y=a}^{a+k-1}\gcd(z_i - y, z_j - y)\right)
        = v_q\big(\gcd(z_i - y_0, z_j - y_0)\big)
        \le v_q(z_i - z_j). \qedhere
    \]
\end{proof}

\begin{proof}[Proof of \cref{lem:R_k_lower_bound}]
    Since $P = R_{\ge k}(P)\,R_{<k}(P)$, the claim is equivalent to
    $R_{<k}(P) \le k!\,x^{\pi(k)}$. From $P = k!\binom{x}{k}$ we have
    $v_p(P) = v_p(k!) + v_p\binom{x}{k}$ for every prime $p$, and by \cref{lem:lcm_bound},
    $\binom{x}{k} \mid \operatorname{lcm}(x-k+1,\dots,x)$, so
    $v_p\binom{x}{k} \le \max_{0\le i<k} v_p(x-i)$. Thus $p^{\,v_p\binom{x}{k}}$ divides some
    $x - i \le x$, giving $p^{\,v_p\binom{x}{k}} \le x$ and hence
    $p^{\,v_p(P)} \le x\,p^{\,v_p(k!)}$. Multiplying over the at most $\pi(k)$ primes $p<k$ and
    using $\prod_{p<k} p^{\,v_p(k!)} \le k!$,
    \[
        R_{<k}(P) = \prod_{p<k} p^{\,v_p(P)}
        \le x^{\pi(k)}\prod_{p<k} p^{\,v_p(k!)} \le k!\,x^{\pi(k)}. \qedhere
    \]
\end{proof}

\subsubsection{Conclusion}

We bound the running time of  \cref{alg:algorithm_verify_large_b} step by step; the elimination loop at line~\ref{second loop} is the only
nontrivial part, and it is where \cref{lem:survivor_product} is used. 
The loop processes $y = a, \dots, a+2n-1$, and we index its iterations as rounds $r = 1, \dots, 2n$.
The candidates entering round $r$ are exactly those that survived the first $r-1$ rounds; call this
set $S_r$. 

\begin{lemma}\label{lem:round_product}
    For each round $r$, $\prod_{z \in S_r} z \le 2^{\tilde{O}(nb/r)}$.
\end{lemma}
\begin{proof}
    The members of $S_r$ survived $r-1$ rounds, so by \cref{lem:survivor_product} any subset of
    $S_r$ of size at most $(r-1)/2$ has product $2^{\tilde{O}(b)}$. For $r \ge 3$, partition $S_r$ into
    $\lceil |S_r| / \lfloor (r-1)/2 \rfloor \rceil = \tilde{O}(n/r)$ such subsets; multiplying their
    products gives $\prod_{z \in S_r} z \le 2^{\tilde{O}(nb/r)}$. For $r\le 2$, the bound is
    trivial $\prod_z z \le (2^b)^n = 2^{nb}$.
\end{proof}

In round $r$ we compute $f(y) \bmod (z - y)$ for all $z \in S_r$ via \tool{comp:multipoint_modulo},
at cost $\tilde{O}\big(\log |f(y)| + \sum_{z \in S_r} \log(z - y)\big)$. Since $y \le 2n^2 + 2n$ we
have $\log |f(y)| = \tilde{O}(n + b)$, and by \cref{lem:round_product},
$\sum_{z \in S_r} \log(z - y) \le \log \prod_{z \in S_r} z = \tilde{O}(nb/r)$. Hence round $r$ costs
$\tilde{O}(n + b + nb/r)$, and summing over all $2n$ rounds,
\[
    \sum_{r=1}^{2n} \tilde{O}\!\left(n + b + \frac{nb}{r}\right)
    = \tilde{O}\big(n(n+b)\big) + \tilde{O}\!\left(nb \sum_{r=1}^{2n} \tfrac{1}{r}\right)
    = \tilde{O}(n^2 + nb),
\]
the second term collapsing to $\tilde{O}(nb)$ by the harmonic sum. So the loop runs in
$\tilde{O}(n^2 + nb)$ time.

We now consider the remaining steps.
\begin{itemize}
    \item \textbf{Line \ref{step:eliminate_small_x_i}} (verify and strip small candidates). Evaluating $f$ at the $\le n$ small
    candidates simultaneously via \tool{comp:multipoint_evaluation} costs $\tilde{O}(nb)$, since each
    small $x_i$ has $O(\log(n)+b/n)$ bits.
    \item \textbf{Line \ref{step:find_segment}} (find the segment). By \cref{lem:find_segment_f_doesn't_vanish}, this
    takes $\tilde{O}(n^2 + nb)$ time.
    \item \textbf{Line \ref{step:evaluate_segment}} (evaluate $f$ on $[a, a+2n-1]$). The $2n$ points are $O(\log n)$ bits so \tool{comp:multipoint_evaluation} costs
    $\tilde{O}\big(n(n+b)\big) = \tilde{O}(n^2 + nb)$.
    \item \textbf{Line \ref{step:output_candidates}} (output) requires no computation.
\end{itemize}

Summing all steps, \cref{alg:algorithm_verify_large_b} runs in $\tilde{O}(nb + n^2)$ time.

\subsection{Algorithm for small $b$}
For small values of $b$, the large $b$ algorithm is quite inefficient. We therefore give a second algorithm for the same verification problem in $\tilde{O}(nb^2)$ time; this is faster when $b \le \sqrt{n}$.

\begin{algorithm}[H]
    \caption{Verifying roots for small $b$}
    \label{alg:algorithm_verify_roots_small_b}
    \raggedright
    
    Input: positive integers $n,b$, the coefficients $a_0,\ldots,a_n $ of a polynomial $f$ with $a_0 \neq 0$ and $\|f\|_\infty < 2^b$. In addition, we are given $n$ distinct candidate roots $x_1,\ldots,x_n$ such that $1 \le x_i < 2^b$ for $i = 1,\ldots,n$.
    Output: The indices of the $x_i$ which are roots of $f$.
    \noindent\rule{\textwidth}{0.4pt}
    \begin{algorithmic}[1]
        \State Eliminate any candidates that do not divide $a_0$ \label{step:x_i_div_a_0}
        \State Set $M = |a_0|^{b}$ \label{step:calc_M}
        \State Calculate $f(x_i) \bmod M$ for every remaining candidate using \tool{comp:multipoint_evaluation} \label{step:f(x_i)modM}
        \State Eliminate any $x_i$ that yields a non-zero residue mod $M$ \label{step:remove_candidates}
        \State Calculate $f(x_i)$ for each remaining candidate separately using \tool{comp:substitute_in_poly} \label{step:final_check}
        \State\label{step:check_f(x_i)=0} Output all $x_i$ such that $f(x_i)=0$ 
    \end{algorithmic}
\end{algorithm}

\subsubsection{Correctness}\label{sec:verify_small_b_correctness}
It is clear that any $x_i$ that was eliminated is not a root of $f$. Conversely,
every candidate $x_i$ that is output satisfies $f(x_i)=0$ by line \ref{step:check_f(x_i)=0}.

\subsubsection{Complexity analysis}
We prove that the algorithm takes $\tilde{O}(nb^2)$ time.

Since modular arithmetic on integers of bit size $O(b)$ takes $\tilde{O}(b)$ time, line \ref{step:x_i_div_a_0} takes $\tilde{O}(nb)$ time. Using fast exponentiation, line \ref{step:calc_M} takes $O(\log(b))$ multiplications of integers of size $O(b^2)$ bits, for a total of $\tilde{O}(b^2)$ time. By \tool{comp:multipoint_evaluation}, line \ref{step:f(x_i)modM} costs $\tilde{O}(n)$ operations in $\mathbb Z/M\mathbb Z$, each requiring $\tilde{O}(\log(M))=\tilde{O}(b^2)$ time. Thus line \ref{step:f(x_i)modM} runs in $\tilde{O}(nb^2)$. Line \ref{step:check_f(x_i)=0}
requires no computations, so it only remains to bound the running time of line \ref{step:final_check}. We prove the following claim:

\begin{lemma}\label{not_many_candidates_remain}
At most $b+1$ candidates remain when line \ref{step:final_check} is reached.
\end{lemma}
\begin{proof}
    If $b=1$, the claim is trivial, as $1 \le x_i < 2$. Now suppose $b > 1$. Candidates that remain after line \ref{step:remove_candidates} is run will be called survivors. Let $p \mid a_0$ be a prime. By \cref{x_appears_on_lifting_tree}, there are at most $v_p(a_0)$ survivors $x_i$ with $p \mid x_i$. For every survivor $x_i >1$, we have $x_i \mid a_0$, and we associate $x_i$ with one prime divisor $p \mid x_i \mid a_0$. Since each prime $p \mid a_0$ can be associated with at most $v_p(a_0)$ survivors, the number of survivors other than possibly $x_i=1$ is at most $\Sigma_{p \mid a_0}v_p(a_0) \le \log(|a_0|) \le b$.   
\end{proof}

By \tool{comp:substitute_in_poly}, each value $f(x_i)$ can be computed in $\tilde{O}(nb)$ time. Therefore, by the lemma above, line \ref{step:final_check} runs in $\tilde{O}(nb^2)$ time. This proves that the algorithm runs in $\tilde{O}(nb^2)$ time.

\subsection{Conclusion of \cref{thm:verifying_roots}}
 Let $f$ be a polynomial and let $x_1,\ldots,x_m$ be candidate roots. We explain how to use \cref{alg:algorithm_verify_large_b} and \cref{alg:algorithm_verify_roots_small_b} to verify which of the $x_i$ are roots of $f$. 
 
 First, the algorithm assumes $m = n$ only for convenience, and it is clear that the algorithms apply to $m\le n$ as well. Next, we find the largest $k$ such that $x^k \mid f$, and replace $f$ by $\frac{f}{x^k}$. If $0$ is among the candidates, we mark it as a root exactly when $k>0$, and discard it from the list. We now assume $a_0 \neq 0$.

Both algorithms take positive candidates, so we handle the two signs separately. We run the chosen
algorithm once on the positive $x_i$ with $f$, and once on the values $-x_i$ for the negative
$x_i$ with $g(x) = f(-x)$, using that $x_i$ is a root of $f$ iff $-x_i$ is a root of $g$. The
choice of algorithm depends on $b$: if $b \ge \sqrt{n}$ we use \cref{alg:algorithm_verify_large_b},
which runs in $\tilde{O}(nb + n^2)$, and otherwise \cref{alg:algorithm_verify_roots_small_b}, which
runs in $\tilde{O}(nb^2)$. Since $n^2 \le nb^2$ exactly when $b \ge \sqrt{n}$, in both regimes the
cost is $\tilde{O}\big(nb + \min(n^2, nb^2)\big)$, proving \cref{thm:verifying_roots}.

\section{Proof of  \cref{thm:finding_integer_roots}}\label{sec:final-proof}

We now combine everything to prove  \cref{thm:finding_integer_roots}. 

\begin{proof}[Proof of  \cref{thm:finding_integer_roots}]

We now explain how these results imply the root-finding theorem. Set
\(t=\lfloor \sqrt n\rfloor\). By \cref{thm:finding_p}, we may find a prime
\(p\le \tilde{O}(nb/t)\) such that \(p\) does not divide the leading
coefficient of \(f\), and
\[
    \deg(\gcd(f_p,f_p'))<t,
\]
in time
\[
    \tilde{O}\left(\frac{n^2b}{t}\right)
    =
    \tilde{O}(n^{3/2}b).
\]
The quantity \(\deg(\gcd(f_p,f_p'))\) bounds the sum of \(m_\alpha-1\) over the
roots \(\alpha\) of \(f_p\), where \(m_\alpha\) is the multiplicity of
\(\alpha\). Hence the sum of the multiplicities of the repeated roots is at
most \(2t\). Since \(\|f\|_\infty<2^b\le p^b\le p^{2bt}\), we may apply \cref{thm:lifting_roots} with parameter $2t$, which gives a list of at most \(n\)
candidates containing the first \(b\) \(p\)-adic digits of every \(p\)-adic root
of \(f\). Since every integer root of \(f\) has absolute value \(<2^b\), this
list contains every non-negative integer root of \(f\). Using \(t=\lfloor\sqrt n\rfloor\) and
\(p\le \tilde{O}(nb/t) \le \tilde{O}( nb)\), the running time of this lifting step
is
\[
    \tilde{O}\left(nb t+(t+\sqrt n)\sqrt p\right)
    =
    \tilde{O}\left(n^{3/2}b\right).
\]
Finally, we remove any candidate which is at least $2^b$ in absolute value and then use
\cref{thm:verifying_roots} to test which of the candidates are genuine integer
roots. Because \(\min(n^2,nb^2)\le \sqrt{n^2\cdot nb^2} =n^{3/2}b\), this costs
\[
    \tilde{O}(nb+\min(n^2,nb^2))\le \tilde{O}(n^{3/2}b).
\]
This gives us all non-negative roots of $f$. We run this whole procedure again on $f(-x)$, which is also square-free and satisfies $\|f(-x)\|_\infty <2^b$, to find all positive integer roots of $f(-x)$. These are the absolute values of the negative roots of $f$. 

Summing the three stages gives
\[
    \tilde{O}(n^{3/2}b),
\]
as claimed.

\end{proof}

\section{Future work}\label{sec:future_work}
There are several ways to generalize or improve our results.
\begin{enumerate}
    \item Finding rational roots. The methods of this paper
can be extended from integer roots to rational roots. We give the main ideas here. Let \(x/y\) be a
rational root of \(f\), with $\gcd(x,y)=1$. By the rational root theorem,
\(x\) divides the constant coefficient of \(f\), and \(y\) divides the leading
coefficient of \(f\). In particular, we have \(|x|,|y|< 2^b\). Moreover, since the prime \(p\) chosen in \cref{sec:finding-p} does not divide the leading coefficient of \(f\), we have
\(p\nmid y\). Thus \(x/y\) is a \(p\)-adic integer.

Therefore, invoking \cref{thm:lifting_roots} with $2b+2$ in place of $b$ returns an integer \(z\) satisfying
\[
    z \equiv \frac{x}{y} \pmod {p^{2b+2}}
\]
By standard rational reconstruction, one can recover the
rational number \(x/y\) from \(z\), using $|x|,|y| < 2^b$.
Rational reconstruction goes back to Wang's $p$-adic reconstruction algorithm
and can be implemented via the extended Euclidean algorithm
\cite{WangGuyDavenportRationalReconstruction}.

Our root verification algorithms can also be extended to rational numbers. In \cref{alg:algorithm_verify_large_b}, the observation is that if $f(c/d)=0$ then $dy-c \mid f(y)$. Thus the test $x_i-y \mid f(y)$ is replaced by the test $dy-c \mid f(y)$. For \cref{alg:algorithm_verify_roots_small_b}, the test $a_0^b \mid f(x_i)$ is replaced by $a_0^{2b+2} \mid d^nf(c/d)$. We expect the running-time
analyses to carry over with only minor changes, and leave the detailed analysis to future work. 

    \item Solve for non-square-free polynomials. The standard way to deal with such polynomials is finding integer roots for $h = \frac{f}{\gcd(f,f')}$ which is square-free and has the same roots as $f$. However, we are not aware of a way to compute $\gcd(f,f')$ in $\tilde{O}(n^{3/2}b)$. Harvey and Hittmeir have shown how to do this in $\tilde{O}(n^3+n^2b)$ \cite{HarveyHittmeirRPower}.

    \item  Improve the complexity of the algorithm. We have not been able to improve the running time of any stage of the algorithm by more than polylogarithmic factors. It seems difficult to improve the running time of \cref{thm:lifting_roots} for the following reason: the coefficients of $f$ matter modulo $p^{bt}$. Therefore, in some sense, the input size is $\tilde{O}(nbt)$ bits, which is the dominant part of the complexity of the algorithm that proves 
    \cref{thm:lifting_roots}. To improve the complexity of this algorithm one must use the fact that $\|f\|_\infty < 2^b$, which seems challenging.

\end{enumerate}

\phantomsection
\section*{Acknowledgments}
\addcontentsline{toc}{section}{Acknowledgments}

The author wishes to thank his advisor, Prof. Amir Shpilka, for many helpful discussions on this problem and for valuable comments and suggestions on earlier drafts of this paper. The author also thanks Raz Dvora for helpful discussions on the problem.

\paragraph{Use of AI tools.} All ideas, algorithms, and proofs are the author's own. The author used Claude (Anthropic) and ChatGPT (OpenAI) to suggest improvements to wording and grammar, to help prepare the figures, and to help survey related work and check citations. The author reviewed and revised all suggested text and assumes responsibility for all content. 

\bibliographystyle{plain}
\bibliography{references}

\clearpage
\appendix

\section{Analysis of the randomized algorithm}\label{app:randomized}

In this appendix we analyze the randomized algorithm sketched in the introduction. Throughout, $f\in\Z[x]$ is a square-free polynomial of degree $n$ with $\|f\|_\infty<2^b$. A uniformly random prime in an interval $[L,2L]$ is sampled by choosing uniformly random integers in $[L,2L]$ until trial division finds a prime. We use the following lemma.

\begin{lemma}\label{lem:random_prime}
Let $L\ge 17$. Then the interval $[L,2L]$ contains at least $\frac{L}{3\ln L}$ primes. Moreover, for every nonzero integer $Z$, a uniformly random prime in $[L,2L]$ divides $Z$ with probability at most $\frac{3\ln 2\cdot\log|Z|}{L}$.
\end{lemma}

\begin{proof}
By \cite[(3.5) and (3.6)]{RosserSchoenfeld1962}, $\pi(x)>\frac{x}{\ln x}$ for $x\ge 17$ and $\pi(x)<1.25506\,\frac{x}{\ln x}$ for $x>1$. Hence the number of primes in $[L,2L]$ is at least
\[
    \pi(2L)-\pi(L) > \frac{2L}{\ln(2L)}-1.25506\,\frac{L}{\ln L}
    = \frac{L}{\ln L}\left(\frac{2\ln L}{\ln L+\ln 2}-1.25506\right)\ge \frac{L}{3\ln L},
\]
since the expression in parentheses is increasing in $L$, and is larger than $0.35$ for $L=17$.

If $Z$ has $k$ distinct prime divisors $\ge L$, then their product divides $Z$, so $L^k\le|Z|$ and $k\le\frac{\log|Z|}{\log L}$. Therefore, the probability that a uniformly random prime in $[L,2L]$ divides $Z$ is at most
\[
    \frac{\log|Z|}{\log L}\cdot\frac{3\ln L}{L}=\frac{3\ln 2\cdot\log|Z|}{L}. \qedhere
\]
\end{proof}

\begin{theorem}\label{thm:randomized}
There is a randomized algorithm that runs in expected time $\tilde{O}(nb)$, whose output always contains every non-negative integer root of $f$, and with probability at least $15/16$ contains no other integer.
\end{theorem}

We prove \cref{thm:randomized} using the following algorithm.

\begin{algorithm}[H]
    \caption{Randomized root finding}\label{alg:randomized}
    \raggedright
    Input: A square-free polynomial $f\in\Z[x]$ of degree $n$ with $\|f\|_\infty<2^b$.\\
    Output: A set of non-negative integers.
    \noindent\rule{\textwidth}{0.4pt}
    \begin{algorithmic}[1]
        \Repeat
            \State Choose a uniformly random prime $p\in[100(nb+n\log n),\,200(nb+n\log n)]$, and compute $f_p$, $f_p'$ and $\gcd(f_p,f_p')$. \label{step:rand_choose_p}
        \Until{$\gcd(f_p,f_p')=1$}
        \State Compute $f_p(0),f_p(1),\ldots,f_p(p-1)$ via multipoint evaluation using \tool{comp:multipoint_evaluation}, and let $R\subseteq\{0,\ldots,p-1\}$ be the set of roots of $f_p$. \label{step:rand_roots}
        \State For every $r\in R$, compute the unique $x_r\in\{0,\ldots,p^b-1\}$ with $x_r\equiv r \pmod p$ and $f(x_r)\equiv 0\pmod{p^b}$, by Newton iteration as in lines \ref{step:hensel_eval}--\ref{step:hensel_update} of \cref{alg:finish_lifting_roots}. \label{step:rand_lift}
        \State Let $X=\{x_r : r\in R,\ x_r<2^b\}$. \label{step:rand_discard}
        \State Choose a uniformly random prime $q\in[100n^2b,\,200n^2b]$, and compute $f(x)\bmod q$ for all $x\in X$ via multipoint evaluation over $\mathbb F_q$ using \tool{comp:multipoint_evaluation}. \label{step:rand_verify}
        \State Output $\{x\in X : f(x)\equiv 0\pmod q\}$.
    \end{algorithmic}
\end{algorithm}

\paragraph{Correctness.} Let $x\ge 0$ be an integer root of $f$. Then $x<2^b\le p^b$, and $x\bmod p$ is a simple root of $f_p$ since $\gcd(f_p,f_p')=1$. By Hensel's lemma, $x$ is the unique lift of $x\bmod p$ in $\{0,\ldots,p^b-1\}$, so $x\in X$, and since $f(x)=0$ it is output.

Conversely, let $x\in X$ with $f(x)\neq 0$. By \cref{lem:f-ub}, $0<|f(x)|\le (n+1)2^b2^{bn}\le 2^{3nb}$. By \cref{lem:random_prime} with $L=100n^2b$, the probability that $q\mid f(x)$, i.e., that $x$ is output, is at most $\frac{3\ln 2\cdot 3nb}{100n^2b}<\frac{1}{16n}$. Since $|X|\le n$, by the union bound some non-root is output with probability less than $\frac1{16}$.

\paragraph{Running time.} We prove that the algorithm runs in expected time $\tilde{O}(nb)$.

If $p\nmid\operatorname{res}(f,f')$, then $\gcd(f_p,f_p')=1$ by \cref{lem:res_modular_bound}. By \cref{lem:res_size_bound}, $0<|\operatorname{res}(f,f')|\le 2^{6(nb+n\log n)}$, as computed in \cref{lem:not_lots_of_bad_primes}. Thus, by \cref{lem:random_prime} with $L=100(nb+n\log n)$, an iteration of line \ref{step:rand_choose_p} fails with probability at most $\frac{18\ln 2}{100}<\frac18$, and the expected number of iterations is less than $\frac87$.

By \cref{lem:random_prime}, a uniformly random integer in $[L,2L]$ is prime with probability at least $\frac{1}{4\ln L}$, so sampling a uniformly random prime in $[L,2L]$ takes an expected $O(\log L)$ trial divisions, i.e., $\tilde{O}(\sqrt L)$ time. This is $\tilde{O}(\sqrt{nb})$ for $p$ and $\tilde{O}(n\sqrt b)$ for $q$. Since $p,q=O(n^2b)$, an arithmetic operation in $\mathbb F_p$ or $\mathbb F_q$ takes $\operatorname{polylog}(n,b)$ time, which is negligible and absorbed into the $\tilde{O}$ notation.

Each iteration of line \ref{step:rand_choose_p} reduces $f$ modulo $p$ in $\tilde{O}(nb)$ time and computes $\gcd(f_p,f_p')$ in $\tilde{O}(n)$ time by \tool{comp:gcd}. By \tool{comp:multipoint_evaluation}, line \ref{step:rand_roots} takes $\tilde{O}(n+p)=\tilde{O}(nb)$ time. By the running time analysis of \cref{alg:finish_lifting_roots}, line \ref{step:rand_lift} lifts the at most $n$ simple roots in $R$ in $\tilde{O}(nb)$ time. In line \ref{step:rand_verify}, reducing $f$ and the at most $n$ elements of $X$ modulo $q$ takes $\tilde{O}(nb)$ time, and by \tool{comp:multipoint_evaluation} the evaluation takes $\tilde{O}(n)$ time. Altogether, the algorithm runs in expected time $\tilde{O}(nb)$, which proves \cref{thm:randomized}.

\end{document}